\documentclass[11pt,a4paper]{amsart}

\usepackage[utf8]{inputenc}
\usepackage{amsmath,amssymb,amsthm}
\usepackage[hidelinks]{hyperref}
\usepackage{enumitem}
\usepackage{needspace}
\usepackage[margin=1in]{geometry}
\usepackage{mathtools}
\usepackage{amsfonts}
\usepackage{graphicx}
\usepackage{xcolor}
\usepackage{tikz}
\usepackage{mathrsfs}
\usetikzlibrary{decorations.markings}
\usepackage{import}
\usepackage[normalem]{ulem}
\usepackage{etoolbox}

\newcommand{\R}{\mathbb{R}}

\newcommand{\Z}{\mathbb{Z}}
\newcommand{\T}{\mathbb{T}}

\newcommand{\F}{\mathcal{F}}

\newcommand{\tf}{\widetilde{f}}
\newcommand{\tS}{\widetilde{S}}
\newcommand{\tK}{\widetilde{K}}

\newcommand{\Homeo}{\operatorname{Homeo}}
\newcommand{\wt}{\widetilde}

\newcommand{\dd}{\,\mathrm{d}}
\DeclareMathOperator{\diam}{diam}
\newcommand{\Id}{\operatorname{Id}}

\newcommand{\Hy}{\mathbb{H}}

\newcommand{\sing}{\operatorname{Sing}}

\newcommand{\wh}{\widehat}

\newcommand{\Fill}{\textnormal{Fill}}
\newcommand{\rhoe}{\rho^{\textnormal{erg}}}
\newcommand{\Fix}{\textnormal{Fix}}
\newcommand{\Cont}{\textnormal{Cont}}
\newcommand{\slea}{\overline{\mathcal{L}}^s}
\newcommand{\ulea}{\overline{\mathcal{L}}^u}
\newcommand{\ilea}{\overline{\mathcal{L}}^i}

\newcommand{\cc}{\textnormal{cc}}

\newcommand{\Bnd}{\mathrm{UB}}

\newcommand{\Per}{\textnormal{Per}}

\theoremstyle{remark}
\newtheorem{remark}{Remark}[section]

\theoremstyle{definition}
\newtheorem{definition}[remark]{Definition}

\theoremstyle{plain}
\newtheorem{theorem}[remark]{Theorem}
\newtheorem{lemma}[remark]{Lemma}
\newtheorem{proposition}[remark]{Proposition}
\newtheorem*{proposition*}{Proposition}
\newtheorem{corollary}[remark]{Corollary}

\newtheorem{teorema}{Theorem}

\newtheorem{corol}[teorema]{Corollary}

\numberwithin{equation}{section}

\counterwithout{figure}{section}
\makeatletter
\let\@affiliations\@empty
\newcommand{\affiliations}[1]{\gdef\@affiliations{#1}}

\apptocmd{\@setauthors}{%
	\ifx\@affiliations\@empty\else
	\par\vspace{0.8\baselineskip}%
	{\centering\small\@affiliations\par}%
	\fi}
\makeatother

\title[Essential factors for surface homeomorphisms]{Essential factors for surface homeomorphisms}
\author[A. García-Sassi]{Alejo García-Sassi\textsuperscript{1}}
\author[F. Tal]{Fábio Tal\textsuperscript{2}}

\affiliations{%
	\textsuperscript{1}IMERL, Universidad de la República, Uruguay --- IME, Universidade de São Paulo, Brasil. \\
	\emph{Email address:} \texttt{alejog@fing.edu.uy}\\[0.5ex]
	\textsuperscript{2}IME, Universidade de São Paulo, Brasil.
	\emph{Email address:} \texttt{fabiotal@ime.usp.br}%
}
\date{}

\begin{document}

\begin{abstract} 
Let $f$ be a homeomorphism of a closed orientable surface of genus {at least one} that is isotopic to a pseudo-Anosov homeomorphism, possibly relative to a finite invariant set. We construct an intrinsic monotone factor by collapsing the maximal continua whose iterates have uniformly bounded diameter in the universal cover. The quotient is homeomorphic to the original surface and the factor is isotopic to $f$ and has several strongly chaotic dynamical features. It is infinitely continuum-wise expansive, tight, and topologically mixing, and it has a dense family of rotational horseshoes, dense periodic points, and a single topological homoclinic class. The factor can be taken as an area-preserving homeomorphism. The quotient preserves stable translation length on the fine curve graph and the applicable homological rotation sets. When $f$ is $C^\infty$, a companion paper gives a unique measure of maximal entropy for the factor; it has full support and is Bernoulli.
\end{abstract}

\maketitle

\section{Introduction}\label{sec:Intro}

\subsection{A global model for essential surface dynamics}

The study of surface homeomorphisms may use several complementary tools
relating topology to dynamics. The fundamental starting point is the Nielsen--Thurston
classification. It organizes, after reduction along invariant families of curves, the different isotopy classes
through periodic and pseudo-Anosov models.
In a pseudo-Anosov class, the representative has a well-understood
geometric and symbolic structure, and its entropy gives a lower
bound for the entropy of every homeomorphism in the class. Handel's
global-shadowing theorem \cite{handel85shadowing} gives this comparison
a more dynamical form. If a homeomorphism of a closed surface is isotopic
to a pseudo-Anosov map, it has a closed invariant subset whose dynamics
semiconjugates onto that representative. Thus the dynamics forced by
the isotopy class appears as a factor of part of the original system.
The invariant subset may be proper, and the target is prescribed by
the isotopy class.

Other tools describe features of the dynamics that vary within a
single isotopy class. Rotation theory records the asymptotic motion
of orbits around the surface; in the identity isotopy class,
homological rotation sets describe this motion in homology.
The action on the fine curve graph measures how iterates displace
essential curves, and its stable translation length gives a geometric
measure of this displacement. These approaches also interact strongly with
relative isotopy. For identity-isotopic homeomorphisms of a closed
surface of genus at least two, Guih\'eneuf and Militon
\cite{guiheneufmiliton23hyperbolic} show that the existence of a pseudo-Anosov representative relative
to a finite invariant set, nonempty interior of the ergodic homological rotation set, and hyperbolic action on the fine curve graph are equivalent ---an equivalence between the first two follows from \cite{garciaguiheneuflessa23tracking}---. For homeomorphisms of $\T^2$, an analogous relationship between rotation sets and pseudo-Anosov dynamics was already established in \cite{franks89realizing,llibremackay91rotation,boyland09transitivity}, while the analogous fine curve graph assertion follows from {\cite{bowdenhenselmannmilitonwebb22rotation}}.  

These viewpoints suggest seeking a global model that preserves the geometric and rotational
information of the given homeomorphism while displaying strong dynamical features that make it a
useful setting in which to study the information that survives the
passage to the factor. This model should live on
the same surface, retain the isotopy class, and be related to the
original map by a semiconjugacy defined everywhere.

In this paper we construct such a factor for homeomorphisms of closed
orientable surfaces of genus at least one that are isotopic to
pseudo-Anosov maps, possibly relative to a finite invariant set.
For every {such a} homeomorphism $f:S\to S$, we obtain a homeomorphism
$g:S\to S$, called its \emph{essential factor}, and a monotone
surjection $q:S\to S$, homotopic to the identity, satisfying
\[
q\circ f=g\circ q.
\]
The factor $g$ is isotopic to $f$ and the whole construction is intrinsic to
$f$. The fibres of $q$ are the maximal continua whose iterates have
uniformly bounded diameter in the universal cover, in both positive
and negative time. These fibres are cellular, their lifted diameters
admit a uniform bound, and their quotient is again homeomorphic to
$S$. The factor dynamics is therefore determined by the given map
through the geometry of its orbits.

The essential factor has a strong collection of chaotic properties
throughout the surface. It is topologically mixing, its periodic
points are dense, and every nonempty open set contains a rotational
horseshoe for an iterate. Its periodic dynamics belongs to a single
topological homoclinic class in the following precise sense: for
every pair of periodic points, the connected stable set
of one and the connected unstable set of the other have
dense topologically transverse intersections. It is also infinitely
continuum-wise expansive, i.e., every nontrivial continuum has unbounded
dynamical diameter in the universal cover. Moreover, it is tight,
meaning that every nontrivial continuum carries positive entropy. The local horseshoes yield
an atomless invariant probability of full support, and hence $g$ is
conjugate to an area-preserving homeomorphism. Thus, after a change
of coordinates, the essential factor provides a conservative model
with these recurrence, expansiveness and homoclinic connection properties.

These properties are obtained while preserving the 
homological and curve-graph invariants considered above. In the
identity isotopy class, the essential factor has exactly the same
homological rotation set as $f$, including the subset of rotation
vectors realized by ergodic measures. More generally, when the isotopy class is not that of the identity, then for
compatible lifts, it preserves the rotation sets associated to
homological quotients on which the induced action of $f$ is trivial. This includes
the rotational information for maps isotopic to a Dehn
twist. Cellularity of the fibers also allows us to apply the invariance result of
Fanoni, Hensel and Le Roux \cite{fanonihenselleroux26translation}:
the stable translation lengths on the fine curve graphs are
unchanged by the quotient. The essential factor therefore realizes
these invariants within the same isotopy class by a conservative,
mixing homeomorphism whose periodic and heteroclinic dynamics are
dense in the whole surface.

The construction extends the conservative models of Garc\'ia-Sassi
and Tal \cite{garciatal24chaotic} for identity-isotopic torus
homeomorphisms whose rotation set has nonempty interior. While \cite{garciatal24chaotic} used mostly the forcing theory developed in \cite{lecalveztal18forcing, lecalveztal22topological},
here the
generalized invariant foliations developed
by Fathi and Militon \cite{fathi90homotopical,militon24foliations}
provide the geometric structure needed for the quotient and its
dynamical properties. The resulting framework covers both closed
pseudo-Anosov isotopy classes and the relative classes arising in
rotation theory.

The geometry of the quotient also gives information about the
original map. For nonwandering maps isotopic to the identity and
pseudo-Anosov relative to a finite set, all periodic open topological
disks have lifted diameter bounded by a single constant, independent
of their periods. In the smooth setting, the companion theorem of
Crovisier, Garc\'ia, Lima and Tal
\cite{crovisiergarcialimatal26mme} gives further statistical structure:
the essential factor of a $C^\infty$ diffeomorphism has a unique
measure of maximal entropy, with full support, and this measure is
Bernoulli. The semiconjugacy identifies it measure-theoretically with
the maximal-entropy measure of a homoclinic class of the original
diffeomorphism. These consequences illustrate how the essential factor
serves both as a global dynamical model and as a means of recovering
geometric and statistical information from the original system.

\subsection{Main results}
Throughout this section ---unless specified otherwise--- $S$ is a closed orientable surface of genus at least {one},
$\tS$ is its universal cover, and $\Gamma$ is its group of deck
transformations. Fix a compatible deck-invariant proper metric on
$\tS$, {and a homeomorphism $f$ of $S$, with a lift $\tf$ to $\tS$}. For a {continuum $K \subset S$ and a compact lift $\tK$}, set
\[
 {\mathscr D_{f}(K)=\sup_{j\in\mathbb Z}\diam(\tf^j \tK).}
\]
The finiteness of this quantity is independent of the chosen lift.
A continuum downstairs is \emph{dynamically bounded} if it has a
compact lift with finite dynamical diameter. The precise convention
for arbitrary continua is given in Section~\ref{sec:prelim}.
Write $\Homeo_A(S)$ for the {set of} homeomorphisms isotopic to $A$.

We use \emph{infinitely continuum-wise expansive} to mean that every
nontrivial continuum has infinite dynamical diameter {(see Definitions \ref{def:dynamical_diameter}, \ref{def:CW_expansive})}. We use
\emph{tight} in the two-sided entropy sense: every nontrivial continuum
has positive entropy when both positive and negative iterates are
observed {(see Section \ref{subsec:tight})}. 

{
\begin{teorema}[\textbf{The essential factor -- pseudo-Anosov case}]\label{thmA:pA}
	Let $S$ be a closed orientable surface of genus at least two, let $A: S \to S$ be a pseudo-Anosov homeomorphism and let $f \in \Homeo_A(S)$. Then, there exists an \emph{essential factor} $g \in \Homeo_A(S)$ and a {monotone semiconjugacy homotopic to the identity} $qf = gq$ such that:
	\begin{enumerate}
		\item the fibres of $q$ are cellular continua whose lifted diameters are uniformly bounded;
		\item $g$ is infinitely continuum-wise expansive;
		\item $g$ is tight: $h_{\pm}(g,K) > 0$ for each nontrivial continuum $K \subset S$ (see Section \ref{subsec:tight});
		\item for each {nonempty} open set $U \subset S$, {some positive power of $g$} has a rotational horseshoe $\Lambda \subset U$. In particular, $\overline{\Per(g)} = S$.
		\item $g$ is topologically mixing;
		\item $g$ has a unique homoclinic class: $\forall p_1,p_2\in\Per(g)$, the connected stable and unstable sets
		$\mathcal C^s(p_1),\mathcal C^u(p_2)$
		have dense topologically transverse intersections  (see Definition~\ref{def:transverse_connections});
		\item $g$ is conservative: it is conjugate to a homeomorphism preserving an area measure.
	\end{enumerate}
The decomposition is intrinsic, and taking the essential
factor of $g$ leaves $g$ unchanged. 
\end{teorema}
}

The invariant probability in {(7)} is a consequence of the local
rotational horseshoes. In each member of a countable open basis, such a
system supports an atomless probability invariant under an iterate.
Averaging over finitely many iterates of $g$ and taking a positive
weighted sum gives an invariant probability of full support. The
Oxtoby--Ulam theorem then provides the conservative realization
\cite{oxtobyulam41measure}.

A homeomorphism $f: S \to S$ is \emph{pseudo-Anosov relative to a finite set},
or \emph{PARF}, if it is isotopic, relative to a finite $f$-invariant
set $X$, to a pseudo-Anosov map on $S\backslash X$. {Note that in this context, $S$ can also be of genus one}.

\begin{teorema}[\textbf{The essential factor -- relative case}]\label{thmB:PARF}
{	
Let $f \in \Homeo(S)$ be a PARF homeomorphism. Then, $f$ admits an intrinsic monotone semiconjugacy $qf=gq$, with $q$ homotopic to the identity. The \emph{essential factor} {$g \in {\Homeo}(S)$ is isotopic to $f$ and}
has all the dynamical properties in Theorem~\ref{thmA:pA}(1)--(7). {Taking the essential factor of $g$ leaves $g$ unchanged.}}
\end{teorema}

{
\begin{remark}
	If $f: \T^2 \to {\T^2}$ is isotopic to a linear Anosov map, then $f$ is PARF and thus satisfies Theorem \ref{thmB:PARF}: indeed, $f$ has a fixed point $p$ and it is isotopic, relative to $\{p\}$, to a pseudo-Anosov map on $\T^2 \backslash \{p\}$.	
\end{remark}
}

\begin{remark}\label{rem:parf_vs_id}
	$f\in\Homeo_0(S)$ is PARF if and only if {the ergodic rotation set of a lift $\wt f$ of $f$ to the universal covering has nonempty interior.}
\end{remark}

Here {the \emph{ergodic rotation set} of a lift $\wt f$ of $f$ consists of the homological rotation vectors for $\wt f$ of the $f$-invariant ergodic measures}. For the torus case, the condition above is equivalent to the usual Misiurewicz-Ziemian rotation set having nonempty interior.  

The next theorem records the geometric and homological information
retained by the quotient. Its first assertion {follows from the proof of the cell-like invariance theorem in \cite[Theorem 1.2(1)]{fanonihenselleroux26translation}, which applies to both the fine curve graph and its nonseparating subgraph.}

\Needspace{6\baselineskip}
\begin{teorema}[\textbf{Invariants preserved by the essential factor}]\label{thmC:rotational}
Let $qf=gq$ be the monotone {semiconjugacy in Theorem} \ref{thmA:pA} {or} \ref{thmB:PARF}. 
\begin{enumerate}
\item The stable translation lengths $\mathrm{tl}(f), \mathrm{tl}(g)$ on {the respective fine curve graphs and their nonseparating subgraphs} \cite{fanonihenselleroux26translation} satisfy $\mathrm{tl}(f)=\mathrm{tl}(g)$; 
\item
If $f\in\Homeo_0(S)$, fix a lift $F$ given by an identity isotopy, and the compatible lift $G$ of $g$ satisfying $\Phi F = G\Phi$. 
Then the homological and ergodic rotation sets satisfy
\[
 \rho_{H_1}(F)=\rho_{H_1}(G),\qquad
 \rhoe_{H_1}(F)=\rhoe_{H_1}(G).
\]
For genus at least two these are the unique lifts commuting
with the deck transformations.
\item Let $L:H_1(S;\mathbb R)\to V$ be a linear quotient with
$L\circ f_*=L$, and choose compatible lifts $F,G$ satisfying
$\Phi F=G\Phi$, where $\Phi$ is the equivariant lift of $q$.
Then
\[
 \rho_L(F)=\rho_L(G),\qquad
 \rho_L^{\mathrm{erg}}(F)=\rho_L^{\mathrm{erg}}(G).
\]
{In particular, if $f$ is isotopic to a Dehn twist, then taking the essential factor preserves the rotational information{.}}
\end{enumerate}
\end{teorema}

\begin{teorema}[\textbf{Uniformly bounded elliptic islands}]
\label{thmD:periodic_islands}

Let $f\in\Homeo_0(S)$ be a nonwandering PARF homeomorphism.
There is $M_f<\infty$ such that every periodic open topological
disk $U\subset S$ has {lifted diameter} at most $M_f$,
independently of its period. Moreover, $\overline U$ is contained
in a fibre of the essential quotient over a periodic point of $g$.
\end{teorema}

The {lifted diameter} is the diameter of any connected lift of $U$.
The proof in Section~\ref{sec:PARF} combines the
invariant-disk theorem of \cite{koropeckital18fully} with the
uniform bound for dynamically bounded continua.

\medskip

Combining Theorems \ref{thmB:PARF} and \ref{thmC:rotational} we obtain the following. 

\begin{corol}[Realization of rotation sets]\label{cor:realization_rotation_sets}
	A subset of $H_1(S,\R)$ with nonempty interior is the ergodic rotation set of a {lift of a} homeomorphism of $S$ if and only if it is the ergodic rotation set of a {lift of a} conservative homeomorphism of $S$.
\end{corol}

The smooth case has a further consequence of the companion manuscript
\cite{crovisiergarcialimatal26mme}. 

\begin{corol}[MME for smooth extensions \cite{crovisiergarcialimatal26mme}]\label{cor:smooth_mme}
If $f$ in Theorem~\ref{thmA:pA} or \ref{thmB:PARF} is a
$C^\infty$ diffeomorphism, its essential factor $g$ has a unique
measure of maximal entropy. This measure has full support and is
Bernoulli. There is a homoclinic class $H$ of $f$ with $q(H)=S$
such that
\[
 h_{\rm top}(g)=h_{\rm top}(f|_H),
\]
and $q$ induces a measure-theoretic isomorphism between their
maximal-entropy measures. 
The conservative realization can be chosen so that the 
area is its unique maximal-entropy measure.
\end{corol}

\begin{remark}
For every $1\leq r<\infty$, Buzzi constructs $C^r$ surface
diffeomorphisms with no measure of maximal entropy \cite{buzzi14nomme}.  The construction from  \cite{crovisiergarcialimatal26mme} shows that the smoothness hypotheses cannot be improved in Corollary~\ref{cor:smooth_mme}.
\end{remark}

{Again in the smooth context, and assuming that the entropies of the map and the pseudo-Anosov representative are equal, Handel's semiconjugacy is global and homotopic to the identity by \cite{handel88entropy} and cellular by \cite[Theorem A]{carrascohertzhertzures26entropy}. Moreover, \cite{garciadecarvalho26factorization} {(in preparation)} identifies the semiconjugacy fibres of the essential factor with
the connected components of Handel's fibres. This implies the following result.}

{
\begin{corol}[Recovery of pseudo-Anosov representative]\label{cor:essential_vs_handel}
	\label{prop:entropy_equality_model}
	Let $f$ be a $C^\infty$ diffeomorphism isotopic to a pseudo-Anosov
	homeomorphism $A$ on the closed surface, and suppose that
	$h_{\rm top}(f)=h_{\rm top}(A)$. Then the fibres of its essential
	quotient $q$ coincide with those of Handel's semiconjugacy onto $A$.
	In particular the essential factor is conjugate to $A$.
\end{corol}
}

\subsection{Outline of the proof}
{The generalized foliation built in \cite{militon24foliations} characterizes} dynamically bounded continua and gives their uniform diameter
bound. The cellular quotient theorem then constructs the essential
factor. Using \cite{kato93fully}, its infinite continuum-wise expansiveness turns every
intersection of a stable and an unstable continuum into a totally disconnected set and yields tightness.

For the recurrence statements, the decisive input is {the topological structure of the generalized foliation}. It first gives density and local {transverse intersections} of the
corresponding {pseudo-leaves}, which {give}{ heteroclinic pseudo-rectangles, as in \cite{garciatal24chaotic}. A Markovian fixed-point argument as in \cite{guiheneuf26theory} gives rotational horseshoes and density of periodic points.}
The same separation argument gives eventual intersections {between} any
nontrivial stable and unstable continua, and hence {topological} mixing.

In the relative case, a finite branched covering {used in \cite{militon24foliations}} replaces one-prong
singularities by ordinary pseudo-Anosov singularities. The central
step is to prove that dynamically unbounded continua upstairs remain
dynamically unbounded after projection. This identifies the upstairs cellular
decomposition with the intrinsic essential {factor} downstairs.

Section~\ref{sec:prelim} develops the leaf-space and planar tools.
Section~\ref{sec:essfactor} constructs the quotient and proves
tightness. Sections~\ref{sec:stableg} and \ref{sec:horseshoes}
establish the stable and unstable structure, {rotational horseshoes}, and
mixing. Section~\ref{sec:conservative} constructs the invariant
probability and proves Theorem~\ref{thmA:pA}.
Sections~\ref{sec:PARF} and \ref{sec:rotational} prove
Theorems~\ref{thmB:PARF} and \ref{thmC:rotational}.

\subsection{Applications and related constructions}
\label{sec:smooth_comparison}
{As noted before,} the essential factor extends the conservative models of
Garc\'ia-Sassi and Tal \cite{garciatal24chaotic} for identity-isotopic
torus homeomorphisms whose rotation sets have nonempty interior.
That construction was inspired by an unpublished preprint of
de Carvalho, Koropecki and Tal. A related approach is the monotone
quotient of de Carvalho and Paternain
\cite{decarvalhopaternain03monotone}, which associates to a
$C^{1+\alpha}$ surface diffeomorphism a tight homeomorphism of a
generalized cactoid, with fibres carrying zero entropy. Our decomposition
is defined by bounded lifted geometry, and its quotient remains a surface.

The relative pseudo-Anosov setting connects the construction to rotation
theory. For identity-isotopic homeomorphisms of a closed surface of genus
at least two, Guih\'eneuf and Militon
\cite{guiheneufmiliton23hyperbolic} characterize this setting by nonempty
interior for the ergodic rotation set, equivalently by
hyperbolic action on the fine curve graph. Cellularity then brings the
translation-length theorem of Fanoni, Hensel and Le Roux
\cite{fanonihenselleroux26translation} into the study of the essential
factor. The uniform fibre bound has a further application to periodic
islands: Theorem~\ref{thmD:periodic_islands} removes dependence on the
period from the invariant-disk bounds of
\cite{koropeckital18fully} under its stated hypotheses. In the torus
case, this was proved in \cite[Theorem E]{garciatal24chaotic}, answering
\cite[Question G]{koropeckital13strictly} in the rotation-interior
setting; the area-preserving $C^{1+\varepsilon}$ case had been treated
by Addas-Zanata \cite{addaszanata15area}.

{When the original map is smooth,} the companion manuscript
\cite{crovisiergarcialimatal26mme} supplies the statistical conclusions
in Corollary~\ref{cor:smooth_mme} showing that the factor has a unique measure of
maximal entropy, with full support, and this measure is Bernoulli. A
complementary result of Carrasco, Rodriguez-Hertz, Rodriguez-Hertz and
Ures \cite{carrascohertzhertzures26entropy} concerns a $C^\infty$
diffeomorphism whose entropy equals that of its closed pseudo-Anosov
representative. Under this equality, Handel's semiconjugacy is global
\cite{handel88entropy}; they prove that its fibres are cellular and
that it is an isomorphism for the distinguished maximal-entropy
measures. In this case the essential factor recovers the pseudo-Anosov
representative, as stated in Corollary~\ref{cor:essential_vs_handel}.

A geometric counterpart appears in factors of geodesic flows.
For compact higher-genus surfaces without focal points, and more
generally without conjugate points and with continuous Green bundles,
Gelfert and Ruggiero
\cite{gelfertruggiero19expansive,gelfertruggiero23green} construct
time-preserving expansive factors with local product structure by
identifying synchronized bi-asymptotic geodesics. In nonpositive
curvature, the nontrivial fibres come from flat strips, with the flow
coordinate fixed by horospheres. Thus the identification collapses the
transverse width of a strip while preserving the time parametrization.
This parallels the use of bounded two-sided lifted geometry in the
essential quotient. Mamani and Ruggiero \cite{mamaniruggiero23factors}
extend the expansive-factor construction to compact manifolds without
conjugate points and with visibility universal cover; uniqueness of the
maximal-entropy measure for the original flow additionally requires an
entropy gap.

These geodesic-flow factors also support thermodynamic applications.
Gelfert, Kwietniak and Lima \cite{gelfertkwietniaklima25equilibrium}
combine them with synchronizing time changes to study equilibrium
states, including those for scaled geometric potentials on rank-one
surfaces of nonpositive curvature. In the visibility setting, Mamani
\cite{mamani25equilibrium} obtains uniqueness for Bowen potentials
constant on the synchronized classes and satisfying a weak pressure
gap. For the essential factor, the smooth conclusions stated above
come from the companion theorem; bounded lifted fibres alone are not
an entropy-preservation hypothesis.

{\subsection{AI use acknowledgment}
This work was made using AI resources to help verify proofs and redaction. While the main arguments were developed without AI, several suggestions streamlining proofs were adopted.
}

	\section*{Acknowledgements}

The first author would like to thank Emmanuel Militon and Favio Pirán for helpful conversations.
\footnote{The first author was partially supported by FAPESP, IMERL and PEDECIBA. The second author was partially supported by CNPq and FAPESP.}  

\section{Preliminaries and notation}\label{sec:prelim}

Consider a closed orientable hyperbolic surface \(S\) of genus $g \geq 2$ and denote by $\tS$ its universal cover, conformally equivalent to the hyperbolic disk $\Hy^2$. Write $\widetilde\pi:\widetilde S\to S$ for the covering projection. Until further notice (see Remark \ref{conv:flatmetric} below, where we will fix a dynamically adapted metric once and for all), $\dd$ denotes a distance on \(S\) and $\wt \dd$ its lift to $\tS$ so that $(\tS,\wt \dd)$ is isometric to $\Hy^2$. As a hyperbolic space, $\tS$ has a well-defined boundary at infinity $\partial\tS$,  homeomorphic to the circle. We write $\overline{S} = \tS \cup \partial \tS$ for the closed disk compactification. The group of covering transformations \(\Gamma \simeq \pi_1(S)\) acts freely and discretely on \(\wt{S}\) by orientation-preserving isometries. Any homeomorphism $\wt h$ of $\tS$ which is a lift of a homeomorphism of $S$ extends continuously to $\overline{S}$, and two isotopic lifts (lifts of isotopic maps obtained by lifting the isotopy) agree on $\partial \tS$.

\subsection{Topological and metric notions}

\begin{definition}[{Continuum -- Connected components}]\label{def:continuum}
	For a metric space $(X,\dd)$, we say a compact connected subset $K \subset X$ is a \emph{continuum}, and it is \textit{nontrivial} when it is not a point. We denote by $\Cont (X)$ the space of continua of $X$.
	For $Y \subset X$, we denote by $\cc (Y)$ the number of connected components of $Y$.
\end{definition}

We will use the following result from continuum theory, \cite[Theorem 5.6]{nadler92continuum}, adapting the notation to our context.

\begin{lemma}[{Boundary bumping Theorem II, \cite{nadler92continuum}}]\label{lem:nadler}
	Let $S$ be a surface, let $K \subset S$ be a continuum, let $U \subset S$ be an open set such that $K \cap U \neq \varnothing$, $K \cap U^{\mathrm{C}} \neq \varnothing$, and let $z \in K \cap U$ be a point. Let $X$ be the connected component of $K \cap U$ which contains $z$.
	Then, $\overline{X} \cap \partial U \neq \varnothing$.
\end{lemma}

We will also use the following classical consequence of the Jordan curve theorem.

\begin{lemma}\label{lem:crossedcontinua}
	Let {$D \subset S$} be a closed topological disk and let $K, C \subset D$ be disjoint continua. Let $a, b \in K \cap \partial D$ and $c, d \in C \cap \partial D$. Then the pairs $\{a,b\}$ and $\{c,d\}$ are not interleaved in $\partial D$: $c$ and $d$ lie in the same component of $\partial D \setminus \{a,b\}$.
\end{lemma}

\begin{definition}[{Inessential set}]\label{def:inessential}
	Given an open set $U \subset S$, we will say that $U$ is \textit{inessential} if every loop in $U$ is homotopically trivial in $S$. Otherwise, we will say it is essential.
	We will say an arbitrary set $X \subset S$ is \textit{inessential} if it has an inessential neighbourhood. Otherwise we will say it is essential.
\end{definition}

Note that for an inessential connected set $D\subset S$, every lift of $D$ (which is a connected component of $\wt \pi^{-1}(D)$) is homeomorphic to $D$. Moreover, a lift $\wt D$ of an inessential connected set $D \subset S$ satisfies $T (\wt D) \cap \wt D = \varnothing$ for any nontrivial deck transformation $T$.

\begin{definition}[{Lifted diameter}]\label{def:lifted_diameter}
	Given a continuum $K \subset S$, we define its \textit{lifted diameter} $\textnormal{diam}(K)$ as the diameter of any connected component $\wt K$ of its lift to the universal covering $\wt S$ of $S$.
\end{definition}

The lifted diameter of $K$ is finite if and only if $K$ is
inessential. {Unless specified}, the {notation} $\diam(X)$ denotes the diameter of a
subset {$X$} of the universal cover, or the lifted diameter just defined
for a continuum {in $S$}.

\begin{definition}[{Dynamical diameter}]\label{def:dynamical_diameter}
	Let $h$ be a homeomorphism of a metric space $X$, and let $K \subset X$ be a continuum. We define the \emph{dynamical diameter}, the \emph{future dynamical diameter}, and the \emph{past dynamical diameter} of $K$, respectively, by
	\[
		\begin{aligned}
			\mathscr{D}_h(K) &:= \sup_{j \in \mathbb{Z}} \diam(h^j(K)), \\
			\mathscr{D}^+_h(K) &:= \sup_{j \geq 0} \diam(h^j(K)), \\
			\mathscr{D}^-_h(K) &:= \sup_{j \leq 0} \diam(h^j(K)).
		\end{aligned}
	\]
\end{definition}

We say a continuum $K\subset X$ is \textit{dynamically bounded} (for $h$) when $\mathscr{D}_h(K) < \infty$, and \textit{dynamically unbounded} otherwise. For a homeomorphism $f$ of $S$ and a continuum $K \subset S$, the dynamical diameter is computed with the lifted diameter, equivalently for a lift $\tf$ and a lift $\wt K$.

\begin{definition}[{Filling}]
	For a continuum $\wt K \subset \wt S$, we define its \emph{filling} $\Fill(\wt K)$ as the union of $\wt K$ with every bounded connected component of its complement. We will say that $\wt K$ is \emph{filled} when $\Fill(\wt K) = \wt K$.
\end{definition}

{The next assertion follows from $S$ being compact.} 

\begin{lemma}\label{lem:uniform_fill}
Let $d$ be a proper deck-invariant metric on $\tS$ compatible with its
topology. For every $R<\infty$ there is $J(R)<\infty$ such that every
continuum ${\tK}\subset\tS$ with $\diam_d {\tK}\leq R$ satisfies
$\diam_d\Fill({\tK})\leq J(R)$. Filling commutes with plane homeomorphisms. 
\end{lemma}

\begin{lemma}\label{lem:large_compact_pieces}
If $E\subset\tS$ is closed, connected and unbounded, then for every
$x\in E$ and every $R>0$ the component of $x$ in
$E\cap\overline B(x,R)$ meets $\partial B(x,R)$. In particular $E$
contains continua of arbitrarily large diameter.
\end{lemma}
\begin{proof}
The one-point compactification $E\cup\{\infty\}$ is a continuum.
Apply boundary bumping to the proper closed subset
$E\cap\overline B(x,R)$ of this continuum. Its relative boundary is
contained in $E\cap\partial B(x,R)$.
\end{proof}

\begin{definition}[{Continuum-wise expansive}]\label{def:CW_expansive}
	A homeomorphism $h$ of a metric space $X$ is \emph{continuum-wise expansive} if there exists $\varepsilon>0$ such that, for any nontrivial continuum $K \subset X$, we have that $\mathscr{D}_h(K) > \varepsilon$. Moreover, we say that $h$ is \emph{infinitely continuum-wise expansive} if, for any nontrivial continuum $K\subset X$, we have that $\mathscr{D}_h(K) = \infty$.
\end{definition}

Continuum-wise expansive maps have several strong properties. For instance, {Kato} proves in \cite[Theorem 1.6]{kato93fully} that every point has nontrivial stable and unstable continua, whose diameters are uniformly bounded from below.

\subsection{Pseudo-Anosov homeomorphisms}\label{subsec:pA}

By Thurston's classification theorem for surface homeomorphisms (see \cite{thurston88geometry, fathilaudenbachpoenaru79travaux, farbmargalit12primer}), a homeomorphism of a closed surface is isotopic to either a periodic, a reducible, or a pseudo-Anosov homeomorphism.

Let $A:S \to S$ be a pseudo-Anosov homeomorphism of a closed hyperbolic surface. We recall its geometric and dynamical structure, following \cite{militon24foliations}.

\smallskip

\paragraph{\textbf{Invariant foliations and transverse measures.}}
The homeomorphism $A$ preserves two transverse singular foliations $\mathcal{F}^s$ and $\mathcal{F}^u$, the \textit{stable} and \textit{unstable} foliations. These foliations have finitely many {prong-type} singularities, each with at least 3 prongs. Following \cite{militon24foliations}, a \emph{leaf} of $\mathcal{F}^i$ ($i = s,u$) is either a regular leaf, or the union of a singularity with all {its branches}. With this convention, distinct leaves are disjoint and every leaf of $\wt{\F}^i$ separates $\tS$.

Lifting to the universal cover $\wt S \simeq \mathbb{H}^2$, we obtain singular foliations $\wt{\mathcal{F}}^s$ and $\wt{\mathcal{F}}^u$, endowed with respective transverse measures $\mu^u$ and $\mu^s$: $\mu^u$ measures the crossing of stable leaves (it vanishes on arcs contained in stable leaves), and $\mu^s$ measures the crossing of unstable leaves. There exists $\lambda \in (0,1)$ such that for every arc $\wt \alpha \subset \wt S$,
\begin{equation}\label{eq:transverse_measures}
	\mu^u(\wt A^{-n}(\wt \alpha)) = \lambda^n \mu^u(\wt \alpha),
	\quad
	\mu^s(\wt A^{n}(\wt \alpha)) = \lambda^n \mu^s(\wt \alpha).
\end{equation}
Both measures are invariant under the action of $\Gamma$.

\begin{definition}[{Stable and unstable pseudodistances}]\label{def:pseudodistance}
	We define the \textit{stable pseudodistance} $\dd^s$ and the \textit{unstable pseudodistance} $\dd^u$ on $\wt S$ by
	\[
	\dd^s(\wt z,\wt z') = \inf \mu^s(\wt \alpha), \quad \dd^u(\wt z,\wt z') = \inf \mu^u(\wt \alpha),
	\]
	where the infimum is taken over arcs from $\wt z$ to $\wt z'$.
\end{definition}

By construction, these pseudodistances satisfy:
\begin{itemize}
	\item $\dd^u(\wt z,\wt z')=0$ if and only if $\wt z$ and $\wt z'$ belong to the same stable leaf;
	\item $\dd^s(\wt z,\wt z')=0$ if and only if $\wt z$ and $\wt z'$ belong to the same unstable leaf;
	\item the function $\dd^s+\dd^u$ defines {a} metric on $\wt S$, compatible with its topology;
	\item both $\dd^s$ and $\dd^u$ are invariant under deck transformations;
	\item rephrasing \eqref{eq:transverse_measures},
	\[
	\dd^u(\wt A^{-1}(\wt z),\wt A^{-1}(\wt z')) = \lambda\, \dd^u(\wt z,\wt z'),
	\quad
	\dd^s(\wt A(\wt z),\wt A(\wt z')) = \lambda\, \dd^s(\wt z,\wt z').
	\]
\end{itemize}

\begin{remark}\label{conv:flatmetric}
In the closed pseudo-Anosov case we set $\wt\dd=\dd^s+\dd^u$ and use
its quotient metric on $S$. In particular $\wt\dd$ is proper and is quasi-isometric to a lifted
hyperbolic metric.
\end{remark}

\begin{definition}[{Stable and unstable diameters}]\label{def:pseudodiameter}
	Given a set $\wt Z \subset \wt S$, we define its \textit{stable diameter} $\diam^s(\wt Z)$ and its \textit{unstable diameter} $\diam^u(\wt Z)$ as
	\[ \diam^s(\wt Z) = \sup \limits_{\wt z, \wt z' \in \wt Z} \dd^s(\wt z, \wt z'), \quad \diam^u(\wt Z) = \sup \limits_{\wt z, \wt z' \in \wt Z} \dd^u(\wt z, \wt z'). \]
\end{definition}

By definition of $\wt \dd$, a subset of $\tS$ is bounded if and only if both its stable and unstable diameters are finite, and $\max (\diam^s(\wt Z), \diam^u(\wt Z)) \leq \diam(\wt Z) \leq \diam^s(\wt Z) + \diam^u(\wt Z)$; in particular $\diam^u(\wt Z_n) \to \infty$ (or $\diam^s(\wt Z_n) \to \infty$) implies $\diam(\wt Z_n) \to \infty$.

\medskip

\paragraph{\textbf{Minimality.}}

A \emph{half-leaf} is the image of a closed ray in a regular leaf, or a prong issuing from a singularity. A \emph{half-leaf segment} is a compact initial segment of a half-leaf.

\begin{lemma}\label{lem:minimal_filling_foliations}
	Let $A: S \to S$ be a pseudo-Anosov homeomorphism. Then the stable and unstable foliations are minimal (every half-leaf is dense in $S$). See \cite{fathilaudenbachpoenaru79travaux, farbmargalit12primer}.
\end{lemma}

\paragraph{\textbf{Separation and sides.}} With the leaf convention above, every leaf $L$ of $\wt{\F}^s$ separates $\tS$: the complement $\tS \setminus L$ has exactly $k$ connected components, where $k$ is the number of ends of $L$ ($k=2$ for a regular leaf). Two disjoint sets each contained in a single complementary component of $L$, and lying in different components, are said to be \emph{separated by $L$}. Since leaves of $\wt \F^s$ are pairwise disjoint, any leaf $L' \neq L$ lies in exactly one component of $\tS \setminus L$. We record the resulting additivity:

\begin{lemma}\label{lem:additivity}
	Let $L, L', L''$ be leaves of $\wt \F^s$ such that $L$ separates $L'$ from $L''$. Then
	$$\dd^u(L',L'') = \dd^u(L',L) + \dd^u(L,L'').$$
	The same identity holds for {points} on the two sides.
\end{lemma}

\begin{proof}
	Any arc from $L'$ to $L''$ must cross $L$; splitting the arc at a crossing point and using the definition of $\dd^u$ gives $\geq$; the triangle inequality gives $\leq$.
\end{proof}

\subsection{The leaf spaces and separation}\label{subsec:leafspaces}

The dual spaces of the lifted measured foliations {are} $\mathbb R$-trees; their metric completions are again $\mathbb R$-trees. Let $\mathcal L^s$, $\mathcal L^u$ be the spaces of leaves of $\wt{\mathcal F}^s$ and $\wt{\mathcal F}^u$. The pseudodistance $\dd^u$ induces a distance on $\mathcal L^s$ (distinct leaves lie at positive $\dd^u$-distance, by the leaf convention which collapses the branches of a singular leaf to a single point of $\mathcal{L}^s$), and similarly $\dd^s$ induces a distance on $\mathcal L^u$. These spaces are not complete: a Cauchy sequence of leaves may converge, as subsets of $\overline{S}$, to a single point of $\partial \tS$. Following \cite{militon24foliations}, we denote by $\slea$ and $\ulea$ the metric completions; for $\xi \in \slea$ we write $\dd^u(\cdot,\xi)$ for the extended distance, and $L(\xi) \subset \overline S$ for the associated limit set, which is either the closure of a leaf or a point of $\partial \tS$.

For $\xi \in \slea \cup \ulea$, the \emph{ends} of $\xi$ are the points of $\overline{L(\xi)} \cap \partial \tS$: a finite set with $k \geq 2$ elements when $\xi$ is the class of a leaf with $k$ ends ($k = 2$ for a regular leaf), and a single point when $L(\xi) \in \partial \tS$.

\begin{lemma}\label{lem:distinct_ends}
	Two distinct leaves of $\wt \F^s$ have disjoint end sets, and the same holds for $\wt \F^u$. Moreover, a leaf of $\wt \F^s$ and a leaf of $\wt \F^u$ have disjoint end sets.
\end{lemma}

The following elementary observation extends the notion of side to completion points.

\begin{lemma}\label{lem:sides}
	Let $M \in \mathcal{L}^s$ and let $\xi \in \slea$ with $\dd^u(\xi, M) > 0$. Then there is a well-defined connected component $\mathrm{side}_M(\xi)$ of $\tS \setminus M$ such that every leaf $L$ with $\dd^u(L,\xi)$ small enough lies in $\mathrm{side}_M(\xi)$. Moreover, if $\xi, \xi' \in \slea$ satisfy $\dd^u(\xi,M)>0$, $\dd^u(\xi',M)>0$ and $\mathrm{side}_M(\xi) \neq \mathrm{side}_M(\xi')$, then
	$$ \dd^u(\xi,\xi') = \dd^u(\xi,M) + \dd^u(M,\xi').$$
\end{lemma}

\begin{proof}
	Let $(L_n)$ be a Cauchy sequence of leaves representing $\xi$, with $\dd^u(L_n, M) \geq \delta > 0$ for $n$ large. If $L_n$ and $L_m$ lay in different components of $\tS \setminus M$, Lemma \ref{lem:additivity} would give $\dd^u(L_n,L_m) \geq 2\delta$, contradicting the Cauchy property; hence the side stabilizes and does not depend on the representing sequence, by the same estimate. The last equality follows from the equality of Lemma \ref{lem:additivity} applied to representing sequences on the two sides, passing to the limit.
\end{proof}

The equality above upgrades to the completion the additivity of Lemma \ref{lem:additivity}; we will use it in this form.

\begin{definition}[{Separation in the completion}]\label{def:separation}
	Let $M \in \mathcal{L}^s$ and $\xi', \xi'' \in \slea \setminus \{[M]\}$. We say that \emph{$M$ separates $\xi'$ from $\xi''$} if $\dd^u(\xi', M) > 0$, $\dd^u(\xi'', M) > 0$ and $\mathrm{side}_M(\xi') \neq \mathrm{side}_M(\xi'')$. The symmetric definition is used in $\ulea$, with $\dd^s$.
\end{definition}

The next lemma supplies a translate of any prescribed ordinary leaf between two distinct {classes}, with uniform relative margins.

\begin{lemma}\label{lem:separating_translate}
	Let $M \in \mathcal{L}^s$ be any leaf and let $\xi', \xi'' \in \slea$ be distinct, with $d := \dd^u(\xi',\xi'') > 0$. Then there exists $\gamma \in \Gamma$ such that $\gamma M$ separates $\xi'$ from $\xi''$, with
	$$ \dd^u(\xi', \gamma M) \geq \tfrac{d}{3}, \qquad \dd^u(\gamma M, \xi'') \geq \tfrac{d}{3}.$$
	The symmetric statement holds for $\ulea$ and leaves of $\mathcal{L}^u$, with $\dd^s$.
\end{lemma}

\begin{proof}
The segment $[\xi',\xi'']$ in the completed dual tree has length $d$.
Choose a nondegenerate subsegment in its open middle third, represented
by a small transverse arc in a nonsingular flow box. Interior points of
a segment joining two completion points belong to the original tree:
the original leaf tree is geodesic and dense in its completion, and any
added completion point is terminal. The transverse arc may be shortened
so that all its {leaf classes} lie in that middle third.

The projected leaf of $M$ is dense, hence crosses the chosen transversal;
lifting the crossing gives a translate $\gamma M$ whose {class} belongs to
the middle third. Removing this {class} separates the endpoints in the
tree, equivalently it puts the endpoint leaves, or their approximating
leaves, on opposite sides of $\gamma M$. The distances to both endpoints
are at least $d/3$. Singular leaves may meet the transversal: their
{classes} are not removed from a whole transverse interval. Only the
chosen flow box is required to contain no singular point.
\end{proof}

\subsection{Fathi--Militon generalized foliations}\label{subsec:militon}

Let $f:S \to S$ be a homeomorphism isotopic to a pseudo-Anosov map $A: S \to S$ -- which we denote $f \in \Homeo_A(S)$ -- and let $\wt f, \wt A$ be isotopic lifts to $\wt S$. We recall the construction of stable and unstable sets for $f$ by Militon in \cite{militon24foliations}, which behave like the respective foliations of $A$.

\begin{definition}[{Stable and unstable sets}]\label{def:stable_set}
	For $\wt z \in \wt S$, we define its \emph{stable set} $W^s(\wt z)$ and its \textit{unstable set} $W^{u}(\wt z)$ as
	\[ W^s(\wt z) = \{\wt z' \in \wt S \ : \ \exists r > 0 \text{such that } \dd(\wt f^{j}(\wt z), \wt f^{j}(\wt z')) \leq r, \ \forall \ j \geq 0\},\]
	\[W^u(\wt z) = \{\wt z' \in \wt S \ : \ \exists r > 0 \text{such that } \dd(\wt f^{j}(\wt z), \wt f^{j}(\wt z')) \leq r, \ \forall \ j \leq 0\}.\]
\end{definition}

We call the elements of $\slea$ and $\ulea$ \emph{stable classes} and \emph{unstable classes}, respectively. We denote them by $\xi$, or by $\xi^s,\xi^u$ when clarification is necessary. Note that $\overline{L(\xi)} \subset \overline{S}$ is either a continuum with finitely many points (\emph{ends}) in $\partial \wt S$, or a point in $\partial \wt S$.

\begin{definition}[{Leaf neighbourhood}]
	For $L \in \mathcal{L}^s$ and $c > 0$, we define
	$$L_c = \{\wt z \in \wt S \ : \ \dd^{u} (L,\wt z) \leq c \}. $$
	For leaves $L' \in \mathcal{L}^u$, we define $L'_c$ as
	$$L'_c = \{\wt z \in \wt S \ : \ \dd^{s} (L',\wt z) \leq c \}. $$
\end{definition}

\begin{theorem}[\textbf{Generalized foliation, \cite{militon24foliations}}]\label{thm:militon_partition}
	Let $f \in \Homeo_A(S)$, and take two isotopic lifts $\wt f, \wt A: \wt S \to \wt S$. Then, for $i=s,u$ (resp. $i' =u,s$), there exist two maps $$\overline{\eta}^i : \ilea \to \Cont (\overline{S}), \quad \eta^{i}: \ilea \to \mathcal{P}(\wt S), \qquad \eta^{i}(\xi) := \overline{\eta}^{i}(\xi) \cap \wt S, $$
	such that the following properties hold:

	\medskip
	\noindent\textbf{\emph{Topology}}
	\begin{enumerate}[label=\emph{(T\arabic*)}, leftmargin=2.5em, itemsep=0pt]
		\item The nonempty sets of each family $\{\eta^i(\xi)\}_{\xi \in \ilea}$ yield a partition of $\wt S$ into closed sets, and $\eta^i(L) \neq \varnothing$ for every $L \in \mathcal{L}^i$; \hfill{\emph{\textbf{Partition}}}
		\item For every $L \in \mathcal{L}^{i}$, we have $\cc( \overline{S} \setminus \overline{\eta}^{i} (L)) = \cc( \overline{S} \setminus \overline{L})$. \hfill{\emph{\textbf{Complement}}} \\ For every $\xi \in \ilea \setminus \mathcal{L}^{i}$, we have $\cc( \overline{S} \setminus \overline{\eta}^{i} (\xi)) = 1$.
		\item For every $\xi \in \ilea$, we have $\overline{\eta}^{i} (\xi) \cap \partial \wt S = \overline{L(\xi)} \cap \partial \wt S$.  \hfill \emph{\textbf{Ends}}
		\item For every $\xi \in \ilea$ such that $\eta^{i}(\xi) \neq \varnothing$, and every neighbourhood $U$ of \\ $\overline{\eta}^{i}(\xi)$ in $\overline{S}$, there exists $\delta > 0$ such that $\overline{\eta}^{i} (\xi') \subset U$ whenever $\dd^{i'}(\xi, \xi') < \delta$.  \hfill \emph{\textbf{Continuity}}
	\end{enumerate}

	\medskip
	\noindent\textbf{\emph{Dynamics}}
	\begin{enumerate}[label=\emph{(D\arabic*)}, leftmargin=2.5em, itemsep=0pt]
		\item For every $\xi \in \ilea$, we have $\wt f (\overline{\eta}^{i} (\xi)) = \overline{\eta}^{i} (\wt A(\xi))$. \hfill \emph{\textbf{Set-to-set conjugacy}}
		\item For every $\xi \in \ilea$ and $\gamma \in \Gamma$, we have $\gamma(\overline{\eta}^{i} (\xi)) = \overline{\eta}^{i} (\gamma (\xi))$. \hfill{\textbf{\emph{Equivariance}}}
		\item For every $\wt z \in \eta^{s} (\xi)$, we have $W^s (\wt z) = \eta^{s} (\xi)$, and \\ for every $\wt z \in \eta^{u} (\xi)$, we have $W^u (\wt z) = \eta^{u} (\xi)$. \hfill\textbf{\emph{Stability}}
	\end{enumerate}
\end{theorem}

By (T1 -- Partition) and (D3 -- Stability), every point $\wt z \in \tS$ determines classes $\xi^s(\wt z) \in \slea$ and $\xi^u(\wt z) \in \ulea$ with $\wt z \in \eta^s(\xi^s(\wt z)) \cap \eta^u(\xi^u(\wt z))$, and by \cite[Remark 4]{militon24foliations} the \emph{class map}
\begin{equation}\label{eq:classmap}
	\wt z \mapsto (\xi^s(\wt z), \xi^u(\wt z)) \in \slea \times \ulea
\end{equation}
is continuous. {The elements $\xi^s(\wt z)$ and $\xi^u(\wt z)$ are the stable and unstable classes of $\wt z$. We call $\overline\eta^s(\xi)$ a \emph{stable pseudo-leaf} with class $\xi$, and use the analogous unstable terminology. When the ambient space is $\tS$, we also call $\eta^s(\xi)$ a pseudo-leaf; this set may not be connected. In particular, $W^s(\wt z)=\eta^s(\xi^s(\wt z))$ is the stable pseudo-leaf through $\wt z$.} By (T3 -- Ends), for $\xi \in \mathcal{L}^i$ the closure $\overline{\eta}^i(\xi)$ meets $\partial \tS$ in at least two points; hence in this case $\eta^i(\xi)$ is unbounded and $\diam(\eta^i(\xi)) = \infty$. More precisely, every connected component of a {pseudo-leaf} is
unbounded: apply boundary bumping to the open subset
$\overline\eta^i(\xi)\cap\tS$ of the continuum
$\overline\eta^i(\xi)$. The closure of each component meets the finite
boundary end set. Components are closed in $\tS$, so
Lemma~\ref{lem:large_compact_pieces} applies to these unbounded components.

The continuity of the class map turns the separation tools of Section \ref{subsec:leafspaces} into crossing tools for continua.

\begin{lemma}\label{lem:intermediate}
	Let $C \subset \tS$ be a continuum and $\wt z', \wt z'' \in C$. Let $M \in \mathcal{L}^s$ be a leaf which separates $\xi^s(\wt z')$ from $\xi^s(\wt z'')$ (Definition \ref{def:separation}).
	Then $C \cap \eta^s([M]) \neq \varnothing$, where $[M] \in \slea$ is the class of $M$. The symmetric statement holds for unstable classes and leaves of $\mathcal{L}^u$.
\end{lemma}

\begin{proof}
	Consider the continuous function $\varphi : C \to [0,\infty)$, $\varphi(\wt z) = \dd^u(\xi^s(\wt z), [M])$ (continuous as the composition of \eqref{eq:classmap} with the $1$-Lipschitz function $\dd^u(\cdot,[M])$). If $\varphi(\wt z_*) = 0$ for some $\wt z_* \in C$, then $\xi^s(\wt z_*) = [M]$ because $\slea$ is a metric space, hence $\wt z_* \in \eta^s([M])$ and we are done.

	Suppose then that $\varphi > 0$ on $C$. Since $C$ is compact, $m := \min_C \varphi > 0$. By Lemma \ref{lem:sides}, the side function $s(\wt z) := \mathrm{side}_M(\xi^s(\wt z))$ is well defined on $C$, and it is locally constant: if $s(\wt z_1) \neq s(\wt z_2)$ then $\dd^u(\xi^s(\wt z_1),\xi^s(\wt z_2)) \geq \varphi(\wt z_1) + \varphi(\wt z_2) \geq 2m$, so by continuity of \eqref{eq:classmap} the function $s$ is constant on sets of small diameter. As $C$ is connected, $s$ is constant on $C$, contradicting $s(\wt z') \neq s(\wt z'')$.
\end{proof}

The following statement comes from the fact that $\wt f$ and $\wt A$ are at finite distance; see \cite[Lemma 6 and its proof]{militon24foliations}.

\begin{lemma}\label{rem:stable_continuum_orthogonally_bounded}
	There exists $C > 0$ such that for every $L \in \mathcal{L}^s$, $L' \in \mathcal{L}^u$,
	$$\eta^s(L) \subset L_C, \quad \eta^u(L') \subset L'_C.$$
\end{lemma}

The point--class estimate below also applies to {completion classes}
and gives the following uniform width bound.

\begin{corollary}\label{cor:stable_cont_orth_bounded}
	There exists $C > 0$ such that for every $\xi \in \slea$, $\xi' \in \ulea$,
	$$ \diam^{u} \eta^s(\xi) \leq 2C, \quad \diam^{s} \eta^u(\xi') \leq 2C.$$
\end{corollary}

{The following consequence compares distances between points with distances between their classes, and follows directly from Lemma \ref{rem:stable_continuum_orthogonally_bounded}.}

\begin{lemma}\label{lem:point_class_distance}
	Let $C$ be the constant of Lemma \ref{rem:stable_continuum_orthogonally_bounded}. For every $\xi \in \slea$, every $\wt w \in \eta^s(\xi)$ and every $\varepsilon > 0$, there is a leaf $L \in \mathcal{L}^s$ with $\dd^u([L], \xi) \leq \varepsilon$ and $\dd^u(\wt w, L) \leq C + \varepsilon$. Consequently, for all $\wt w', \wt w'' \in \tS$,
	$$\dd^u\big(\xi^s(\wt w'), \xi^s(\wt w'')\big) - 2C \leq \dd^u(\wt w', \wt w'') \leq \dd^u\big(\xi^s(\wt w'), \xi^s(\wt w'')\big) + 2C.$$
	The symmetric statements hold for unstable classes, with $\dd^s$.
\end{lemma}

\paragraph{\textbf{Induced actions on the completions.}}
The map $\wt A$ sends leaves of $\wt \F^s$ to leaves and scales the transverse measures by \eqref{eq:transverse_measures}; it therefore induces a map $\wt A_* : \slea \to \slea$ with
$$ \dd^u(\wt A_* \xi, \wt A_* \xi') = \lambda^{-1} \dd^u(\xi, \xi'), $$
defined first on $\mathcal{L}^s$ and extended to the completion by uniform continuity, and symmetrically a map $\wt A_* : \ulea \to \ulea$ scaling $\dd^s$ by $\lambda$. Deck transformations induce isometries $\gamma_*$ of both completions, and $(\gamma \wt A)_* = \gamma_* \wt A_*$. The isotopic lift $\wt f$ does not act on leaves, but it does act on the {pseudo-leaves}, and the two actions are \emph{the same}:

\begin{lemma}\label{lem:induced_actions}
	For every $\xi \in \slea \cup \ulea$,
	$$ \wt f\big(\overline \eta^i(\xi)\big) = \overline\eta^i\big(\wt A_*(\xi)\big), \qquad i = s, u;$$
	that is, $\wt f$ {permutes the pseudo-leaves and induces $\wt A_*$ on their classes}. In particular $\xi^s(\wt f(\wt w)) = \wt A_*(\xi^s(\wt w))$ for every $\wt w \in \tS$. Moreover, the extensions of $\wt f$ and $\wt A$ to $\overline S$ coincide on $\partial \tS$, and the ends of $\wt A_*(\xi)$ are the images of the ends of $\xi$ under this common boundary homeomorphism.
	The same statements hold for the pair $\wt g' := \gamma_0 \tf^n$, $\wt A' := \gamma_0 \wt A^n$, for any $n \geq 1$ and $\gamma_0 \in \Gamma$, with $\wt A'_* = (\gamma_0)_* \wt A_*^n$ scaling $\dd^u$ by $\lambda^{-n}$.
\end{lemma}

\begin{proof}
	The first identity is exactly (D1 -- Set-to-set conjugacy) of Theorem \ref{thm:militon_partition}, {expressed in terms of their classes $\xi$}. Isotopic lifts agree on $\partial \tS$ (Section \ref{sec:prelim}); the statement about ends then follows from (T3 -- Ends):
	$$ \text{ends}(\wt A_* \xi) = \overline\eta^i(\wt A_* \xi) \cap \partial\tS = \wt f\big(\overline\eta^i(\xi)\big) \cap \partial \tS = \wt f(\text{ends}(\xi)) = \wt A(\text{ends}(\xi)).$$
	The statements for $\wt g', \wt A'$ follow by composing with (D2 -- Equivariance).
\end{proof}

{The next result transfers separation from the
foliation to its respective generalized foliation}.

\begin{proposition}\label{lem:persistent_separation}
	Let $M, L', L''$ be three pairwise distinct leaves of $\mathcal{L}^s$, with classes $\xi = [M], \xi' := [L'], \xi'' := [L'']$. Denote by $e_1, \dots, e_k \in \partial \tS$ the ends of $M$ ($k = 2$ if $M$ is regular), and by $I^M_1, \dots, I^M_k$ the connected components of $\partial \tS \setminus \{e_1, \dots, e_k\}$. By Lemma \ref{lem:distinct_ends}, for every leaf $N \neq M$ all the ends of $N$ lie in a single interval $I^M_{j(N)}$. Then the following are equivalent:
	\begin{enumerate}
		\item $M$ separates $\xi'$ from $\xi''$ in $\slea$ (Definition \ref{def:separation});
		\item $L'$ and $L''$ lie in different connected components of $\tS \setminus M$;
		\item $j(L') \neq j(L'')$;
		\item $\eta^s(\xi)$ separates the {pseudo-leaves:} $\eta^s(\xi')$ and $\eta^s(\xi'')$ lie in different connected components of $\tS \setminus \eta^s(\xi)$, and every continuum $K \subset \overline S$ intersecting both $\overline\eta^s(\xi')$ and $\overline\eta^s(\xi'')$ meets $\overline\eta^s(\xi)$.
	\end{enumerate}
	The symmetric statement holds within the unstable family.
\end{proposition}

\begin{proof}
	(1) $\Leftrightarrow$ (2). For a leaf class, $\mathrm{side}_M([L'])$ is computed with the constant representing sequence, so it is the component of $\tS \setminus M$ containing $L'$.

	(2) $\Leftrightarrow$ (3). The closure $\overline M \subset \overline S$ is a continuum meeting $\partial \tS$ exactly at the ends of $M$; its complement $\overline S \setminus \overline M$ has exactly $k$ components, each one the union of a component of $\tS \setminus M$ with the corresponding open boundary interval, and this assignment $V \mapsto I^M_{j(V)}$ is a bijection between the components of $\tS \setminus M$ and the intervals. A leaf $N \neq M$ is a connected subset of $\tS \setminus M$, hence lies in one component $V$, and its ends lie in $\overline{I^M_{j(V)}}$; by Lemma \ref{lem:distinct_ends} the ends of $N$ avoid the ends of $M$, so they lie in the open interval $I^M_{j(V)}$ and $j(N) = j(V)$. Thus $L', L''$ lie in different components if and only if $j(L') \neq j(L'')$.

	(3) $\Leftrightarrow$ (4). By (T2 -- Complement) and (T3 -- Ends), $\overline\eta^s(\xi)$ is a closed connected subset of $\overline S$ meeting $\partial \tS$ exactly at the ends of $M$, and $\overline S \setminus \overline\eta^s(\xi)$ has exactly $k$ components; as before, each component contains exactly one interval $I^M_j$, and the assignment is a bijection. The continuum $\overline\eta^s(\xi')$ is disjoint from $\overline\eta^s(\xi)$, and by (T3 -- Ends) together with Lemma \ref{lem:distinct_ends} its closure meets $\partial \tS$ exactly at the ends of $L'$, which avoid the ends of $M$; hence $\overline\eta^s(\xi')$ lies in the (closure of the) component of $\overline S \setminus \overline\eta^s(\xi)$ whose interval contains the ends of $L'$, namely the one labelled $j(L')$, and similarly for $\xi''$. Therefore the {pseudo-leaves} lie in different components if and only if $j(L') \neq j(L'')$. For the continuum statement, suppose $K \subset \overline S$ is a continuum meeting $\overline\eta^s(\xi')$ and $\overline\eta^s(\xi'')$ and disjoint from $\overline\eta^s(\xi)$. Then
	$$K' := K \cup \overline\eta^s(\xi') \cup \overline\eta^s(\xi'')$$
	is a continuum (the three sets are continua and $K$ meets the other two) disjoint from $\overline\eta^s(\xi)$: the {pseudo-leaves} are pairwise disjoint in $\tS$ by (T1 -- Partition), and their boundary parts are disjoint by (T3 -- Ends) and Lemma \ref{lem:distinct_ends}. Hence $K'$ lies in a single component of $\overline S \setminus \overline\eta^s(\xi)$, so $j(L') = j(L'')$, contradicting (3). Conversely, if (4) holds then the {pseudo-leaves} lie in different components of $\overline S \setminus \overline\eta^s(\xi)$, whose labels are $j(L')$ and $j(L'')$ by the previous paragraph; hence $j(L') \neq j(L'')$, which is (3).
\end{proof}

{Using the dual tree structure and Proposition \ref{lem:persistent_separation}, we obtain the following result.}

\begin{lemma}\label{lem:class_sides}
Let $M$ be a stable leaf with $k$ ends. The preimage under $\xi^s$
of each component of $\slea\setminus\{[M]\}$ is precisely one
component of $\tS\setminus\eta^s([M])$. Its compactified side is
labelled by the corresponding boundary interval of $M$.
The analogous statement holds in the unstable family.
\end{lemma}

\begin{corollary}\label{cor:two_classes_unbounded}
	Let $\wt K \subset \tS$ be a continuum meeting two distinct stable {pseudo-leaves} $\eta^s(\xi') \neq \eta^s(\xi'')$. Then $\sup_{j \geq 0}\diam(\tf^j(\wt K)) = \infty$; quantitatively, with $e := \dd^u(\xi', \xi'') > 0$,
	$$ \diam\big(\tf^j (\wt K)\big) \geq \lambda^{-j} e - 2C \xrightarrow[j \to +\infty]{} \infty.$$
	Symmetrically, a continuum meeting two distinct unstable {pseudo-leaves} satisfies
\[
 \diam(\tf^{-j}(\wt K)) \geq \lambda^{-j} e' - 2C \longrightarrow\infty
 \quad (j\longrightarrow+\infty),
\]
where $e'$ is the $\dd^s$-distance between their {classes}.
\end{corollary}

\begin{proof}
	Take $\wt z' \in \wt K \cap \eta^s(\xi')$ and $\wt z'' \in \wt K \cap \eta^s(\xi'')$. By Lemma \ref{lem:induced_actions}, $\xi^s(\tf^j \wt z') = \wt A_*^j \xi'$ and $\xi^s(\tf^j \wt z'') = \wt A_*^j \xi''$, at distance $\dd^u(\wt A_*^j \xi', \wt A_*^j\xi'') = \lambda^{-j} e$. By Lemma \ref{lem:point_class_distance},
	\[ \diam(\tf^j \wt K) \geq \dd^u(\tf^j \wt z', \tf^j\wt z'') \geq \lambda^{-j} e - 2C. \qedhere\]
\end{proof}

Let $C_0 := \sup_{\wt z} \max \big( \dd(\tf(\wt z), \wt A(\wt z)), \dd(\tf^{-1}(\wt z), \wt A^{-1}(\wt z)) \big) < \infty$ (\cite[Claim 5]{militon24foliations}), and set
\begin{equation}\label{eq:kappa}
	\kappa := \frac{2C_0}{1-\lambda}.
\end{equation}

\begin{lemma}[{\cite[Claim 7]{militon24foliations}}]\label{lem:stable_distance_future_control}
	For every $\wt z, \wt z' \in \tS$ and every $j \geq 0$,
	$$\dd^s(\wt f^j (\wt z),\wt f^j (\wt z')) \leq \kappa + \dd^s (\wt z, \wt z'), \qquad \dd^u(\wt f^{-j} (\wt z),\wt f^{-j} (\wt z')) \leq \kappa + \dd^u (\wt z, \wt z').$$
\end{lemma}

\begin{proof}
	By the triangle inequality, $$\dd^s(\wt f^{j+1} (\wt z),\wt f^{j+1} (\wt z')) \leq 2 C_0 + \dd^s(\wt A \wt f^j (\wt z),\wt A \wt f^j (\wt z')) = 2 C_0 + \lambda \dd^s(\wt f^j (\wt z), \wt f^j (\wt z')).$$
	Since $0 < \lambda < 1$, an inductive argument finishes the proof; the second inequality is symmetric.
\end{proof}

\begin{definition}[Stable and unstable continua]\label{def:stable_continua}
A continuum in $\widetilde S$ is \emph{stable} if it is contained
in one stable {pseudo-leaf}, and \emph{unstable} if it is contained in one
unstable {pseudo-leaf}.
\end{definition}

\begin{lemma}\label{lem:stable_continua_equiv}
	A continuum $\wt K\subset\tS$ is a stable continuum if and only if
\[
 \sup_{j\geq0}\diam(\tf^j\wt K)<\infty.
\]
In that case the supremum is at most $\kappa+\diam^s(\wt K)+2C$.
The symmetric statement holds for unstable continua, using $j\leq0$.
\end{lemma}

\begin{proof}
If $\wt K\subset\eta^s(\xi)$, equivariance and
Corollary~\ref{cor:stable_cont_orth_bounded} bound the unstable
diameter of every iterate by $2C$.
Lemma~\ref{lem:stable_distance_future_control} bounds the stable
diameter at nonnegative times by $\kappa+\diam^s(\wt K)$.
Adding the two estimates gives the bound. Conversely, two points
{with different stable classes} have unbounded forward distance by
(D3 - Stability), which forces the forward dynamical diameter to be infinite.
\end{proof}

\begin{corollary}\label{cor:dynamically_bounded_continua}
	A continuum $\wt K \subset \wt S$ is dynamically bounded if and only if there exists $\wt z \in \wt S$ such that $\wt K \subset W^s(\wt z) \cap W^{u}(\wt z)$. In that case $\diam(\wt K) \leq M_0 := 4C$.
\end{corollary}

\begin{proof}
	If $\wt K$ intersects two different stable (resp. unstable) {pseudo-leaves}, the diameters of its future (resp. past) iterates are unbounded, by (D3 -- Stability) and Corollary \ref{cor:two_classes_unbounded}. Conversely, if $\wt K \subset W^s(\wt z) \cap W^{u}(\wt z)$, then $\wt K$ is simultaneously a stable and an unstable continuum, so by Lemma \ref{lem:stable_continua_equiv} both its forward and backward iterates have bounded diameter. The diameter bound follows from Corollary \ref{cor:stable_cont_orth_bounded}: $\diam(\wt K) \leq \diam^s(\wt K) + \diam^u(\wt K) \leq 2C + 2C$.
\end{proof}

\begin{corollary}\label{cor:hausdorff_limit_stable}
	Let $\{\wt K_n\}_{n \in \Z ^+}$ be a sequence of stable continua with $\sup_n \sup_{j\geq0}\diam(\tf^j (\wt K_n)) \leq M < \infty$, and suppose $\wt K_n \to \wt K$ in the Hausdorff topology. Then $\wt K$ is a stable continuum.
\end{corollary}

\begin{proof}
	Fix $j \geq 0$. By uniform continuity of $\tf^j$ on bounded sets (equivariance plus compactness of $S$), $\tf^j(\wt K_n) \to \tf^j(\wt K)$ in the Hausdorff topology, so $\diam(\tf^j(\wt K)) \leq M$. Hence $\sup_{j \geq 0} \diam(\tf^j(\wt K)) \leq M$ and Lemma \ref{lem:stable_continua_equiv} applies.
\end{proof}

\begin{definition}[Uniform boundedness of dynamically bounded continua]\label{def:UB}
A surface homeomorphism $h$ has property $(\Bnd)$ if there is $M_h<\infty$
such that, for every continuum ${\tK}\subset\tS$,
\[
 \sup_{j\in\mathbb Z}\diam(\wt h^j {\tK})<\infty
 \quad\Longrightarrow\quad \diam {\tK}\leq M_h.
\]
Here the metric is any fixed proper lifted length metric.
The property is independent of the lift and of the choice of such a
metric. Corollary~\ref{cor:dynamically_bounded_continua} proves it in the
closed pseudo-Anosov case, with $M_h=4C$ {for the metric $\wt\dd$ already defined}.
\end{definition}

{Compactness of $S$ and deck equivariance give the following:}

\begin{lemma}\label{lem:powers_UB}
For each $r\in\mathbb Z$ and $R<\infty$ there is $J_r(R)<\infty$ such
that $\diam {\tK}\leq R$ implies $\diam(\wt h^rK)\leq J_r(R)$.
Dynamical boundedness is independent of the chosen lift and is the
same for $h$ and $h^q$, for every $q\geq1$. Property $(\Bnd)$ and the
associated equivalence relation are unchanged by taking powers.
\end{lemma}

\section{The essential factor}\label{sec:essfactor}

The goal of this section is to build the semiconjugacy of Theorems \ref{thmA:pA} and \ref{thmB:PARF}, and to establish its first dynamical consequences. The splitting and quotient arguments use only property $(\Bnd)$ of Definition~\ref{def:UB}. For the closed pseudo-Anosov case it is supplied by Corollary~\ref{cor:dynamically_bounded_continua}; until the later discussion of the Militon structure, $f$ may be any surface homeomorphism with $(\Bnd)$.

\subsection{Dynamically bounded splitting}

As in \cite{garciatal24chaotic}, we partition the surface into dynamically bounded continua.

\begin{definition}[Dynamically bounded equivalence class]\label{def:eq_class}
	Let $\wt z,\wt z' \in \wt S$. We will say that $\wt z \sim_{\wt f} \wt z'$ if there exists a dynamically bounded continuum $\wt K \subset \wt S$ such that $\wt z, \wt z' \in \wt K$.
\end{definition}

It follows immediately that $\sim_{\wt f}$ is an equivalence relation, where we write $[\wt z]_{\wt f}$ for the equivalence class of $\wt z$ under $\sim_{\wt f}$ (or simply $[\wt z]$ and $\sim$ when there is no ambiguity).

\begin{lemma}\label{lem:connected_classes}
	For any $\wt z \in \wt S$, the class $[\wt z]$ is connected.
\end{lemma}

\begin{proof}
	By definition, $[\wt z] = \bigcup_{\wt w \in [\wt z]} \wt K_{\wt w}$, where $\wt z \in \wt K_{\wt w}$ and $\wt K_{\wt w}$ is a continuum for every $\wt w$; a union of connected sets with a common point is connected.
\end{proof}

\begin{lemma}\label{lem:uniformly_bounded_classes}
	There exists $M > 0$ such that $\diam([\wt z]) \leq M$ for every $\wt z \in \wt S$.
\end{lemma}

\begin{proof}
If $x,y\in[\wt z]$, the union of dynamically bounded continua joining
each of them to $\wt z$ is a dynamically bounded continuum. Property
$(\Bnd)$ gives $\dd(x,y)\leq M_f$. Moreover $\tf^j[\wt z]=[\tf^j\wt z]$
for all $j$, so the same bound holds at every iterate.
In the closed pseudo-Anosov case one may take $M_f=4C$.
\end{proof}

\begin{lemma}\label{lem:closed_classes}
	For any $\wt z \in \wt S$, we have $\mathscr{D}_{\tf}([\wt z]) < \infty$, and $[\wt z]$ is closed and filled.
\end{lemma}

\begin{proof}
The closure of $[\wt z]$ is compact and connected. At every iterate its
diameter is at most $M_f$, by the preceding lemma and continuity, so
this closure is itself dynamically bounded and belongs to $[\wt z]$.
Next, filling commutes with iterates, and Lemma~\ref{lem:uniform_fill}
gives the uniform bound $J(M_f)$ for the diameters of the iterates of
$\Fill([\wt z])$. This larger continuum is dynamically bounded and must
also belong to the class.
\end{proof}

\begin{lemma}\label{lem:inessential_classes}
	For any $\wt z \in \wt S$ and any nontrivial deck transformation $T$, we have $T([\wt z]) \cap [\wt z] = \varnothing$.
\end{lemma}

\begin{proof}
Deck transformations preserve dynamical boundedness, since $\tf^jT$ and
$\tf^j$ differ by a deck transformation after composition with $T$.
Thus they permute the equivalence classes. If $T[\wt z]$ meets
$[\wt z]$, the two classes coincide; then the compact set $[\wt z]$
contains the orbit $\{T^k\wt z:k\in\mathbb Z\}$, which is unbounded for
a nontrivial deck transformation of a closed hyperbolic surface. This
is impossible.
\end{proof}

{
\begin{lemma}\label{lem:classes_are_inessential}
	For any $\wt z$, the projection $\wt{\pi}([\wt z])$ is an inessential filled continuum, and the classes are uniformly dynamically bounded.
\end{lemma}

\begin{proof}
	Immediate from Lemmas \ref{lem:connected_classes} through \ref{lem:inessential_classes}.
\end{proof}

\begin{lemma}\label{lem:upper_semi_cont}
	The equivalence relation $\sim_{\wt f}$ is upper semicontinuous.
\end{lemma}

\begin{proof}
	It is enough to prove that the Hausdorff limit $\wt K$ of a sequence of equivalence classes $K_n$ is contained in an equivalence class. Each class is a dynamically bounded continuum, and by (UB) we have that $\diam(\wt f^j \wt K_n) \leq M_f$ for every $j,n \in \Z$. Passing to the limit we obtain that $\diam(\wt f^j \wt K) \leq M_f$ for every $j \in \Z$, which shows that $K$ is contained in an equivalence class. 
\end{proof}
}

\begin{proposition}\label{prop:splitting}
	Let $\sim_f$ be the equivalence relation on $S$ whose classes are $[z]_f = \wt \pi ([\wt z]_{\wt f})$. Then:
	\begin{enumerate}
		\item each equivalence class is an inessential filled continuum;
		\item equivalence classes are uniformly dynamically bounded;
		\item $\sim_f$ is upper semicontinuous.
	\end{enumerate}
\end{proposition}

\begin{proof}
	Items (1) and (2) are Lemma \ref{lem:classes_are_inessential}; item (3) is Lemma \ref{lem:upper_semi_cont}, which descends to $S$ by equivariance.
\end{proof}

\subsection{Moore--Roberts--Steenrod quotient}

Proposition~\ref{prop:splitting} provides an upper semicontinuous
cellular decomposition. Its quotient $S_{\rm ess}=S/{\sim}$ is
homeomorphic to $S$ by the theorem of Roberts and Steenrod
\cite{robertssteenrod38monotone}. This surface theorem extends
Moore's theorem for the sphere \cite{moore25continua}. Let
$Q:S\to S_{\rm ess}$ be the quotient map. Since $f$ permutes the
classes, it induces a homeomorphism $g_{\rm ess}$ satisfying
\[
 Qf=g_{\rm ess}Q.
\]
The map $Q$ is monotone and cellular. This is the higher-genus
version of the construction in \cite[Section 5.2]{garciatal24chaotic}.

\begin{definition}[Essential factor]\label{def:essential_factor}
The quotient dynamics $g_{\rm ess}=f_{/\sim}$ is the
\emph{essential factor} of $f$.
\end{definition}

\begin{remark}\label{rem:canonical_model}
The cellular map $Q$ is a uniform limit of homeomorphisms; see
\cite[Theorem 2.7]{fanonihenselleroux26translation}. Choose a
sufficiently close approximating homeomorphism $Q_0:S\to S_{\rm ess}$.
Close maps into a compact surface are homotopic. Identifying
$S_{\rm ess}$ with $S$ by $Q_0^{-1}$ therefore gives
\[
 q=Q_0^{-1}Q\simeq\Id,\qquad
 g=Q_0^{-1}g_{\rm ess}Q_0,\qquad qf=gq.
\]
We use $q:S\to S$ and $g:S\to S$ for these marked maps throughout.
\end{remark}

\begin{remark}\label{rem:isotopic_factor}
The equation $q f=gq$ and $q\simeq\Id$ give
$g\simeq f$. Homotopic surface homeomorphisms are isotopic, so $g$ has
the same isotopy class as $f$. 
\end{remark}

The marked quotient has a $\Gamma$-equivariant lift $\Phi$ at bounded
distance from the identity. We select the compatible lift of $g$ so that
\begin{equation}\label{eq:Phibounded}
 \sup_{x\in\tS}\dd(\Phi x,x)=B_\Phi<\infty,
 \qquad \Phi\tf=\wt g\Phi.
\end{equation}
This lift is proper and surjective, and its fibres are the lifted
cellular classes. Thus it is monotone and the inverse image of a
continuum is a continuum. 

\subsection{Essential factor dynamics}

{
\begin{lemma}\label{lem:quotient_dynamically_bounded}
	A continuum $K \subset S$ satisfies $\mathscr{D}_f(K) = \infty$ if and only if its projection $qK$ satisfies $\mathscr{D}_g(qK)=\infty$.
\end{lemma}

\begin{proof}
	Immediate from \eqref{eq:Phibounded}: $\wt g^j(\Phi(\wt K)) = \Phi(\tf^j(\wt K))$, and $\Phi$ moves points at most $B_\Phi$.
\end{proof}
}

\begin{lemma}\label{lem:essential_factor_ice}
	Let $f \in \Homeo_A(S)$ and let $g$ be its essential factor. Then $\mathscr{D}_g(K) = \infty$ for every nontrivial continuum $K \subset S$; that is, $g$ is infinitely continuum-wise expansive.
\end{lemma}

\begin{proof}
If a nontrivial continuum $\wt K$ in the quotient cover had
bounded dynamical diameter, its compact connected preimage under
$\Phi$ would be dynamically bounded by the preceding proof. The
preimage would lie in one fibre, making $\wt K$ a point.
\end{proof}

The dynamically bounded splitting of $g$ itself is therefore trivial, and we record:

\begin{remark}\label{rem:idempotent}
	The essential factor of the essential factor of $f$ is the essential factor of $f$. A homeomorphism in $\Homeo_A(S)$ coincides with its essential factor if and only if it is infinitely continuum-wise expansive.
\end{remark}

Note that, since $g \in \Homeo_A(S)$, all of Section \ref{sec:prelim} applies to $g$: we fix from now on the isotopic lifts $\wt g, \wt A$, the associated Militon partitions $\eta^s, \eta^u$ \emph{for $g$}, the constants $C, C_0, \kappa$, and the class maps $\xi^s, \xi^u$, all relative to $g$. By Corollary \ref{cor:dynamically_bounded_continua} applied to $g$ and Lemma \ref{lem:essential_factor_ice}:

\begin{corollary}\label{cor:trivialintersections}
	For the essential factor $g$, every connected component of $\eta^s(\xi) \cap \eta^u(\xi')$ is a single point; in particular the intersection of a stable and an unstable continuum of $g$ is totally disconnected.
\end{corollary}

\subsection{Tightness}\label{subsec:tight}

{
Recall from \cite[Section 1]{garciatal24chaotic} the definition of the two-sided entropy $h_{\pm}(g,K)$ carried by a continuum $K$: it is defined as in Bowen's definition of topological entropy, counting $(\pm n, \varepsilon)$-separated subsets of $K$, where the dynamical distances run over the two-sided window $-n \leq j \leq n$. A homeomorphism is \emph{tight} if $h_\pm(g,K)>0$ for every nontrivial continuum. The following two statements are proved exactly as in \cite[Corollary 5.19 and Proposition 5.20]{garciatal24chaotic}; the proofs are metric-space arguments which use only the infinite continuum-wise expansiveness of $g$ and the compactness of $S$, and we reproduce them for completeness.

\begin{corollary}\label{cor:SplittingContinua}
	Let $\varepsilon>0$ and let $K \subset S$ be a continuum with $\mathrm{diam}(K) > 3 \varepsilon$. Then there exist two subcontinua $K_0, K_1 \subset K$ with $\mathrm{diam}(K_i) \geq \varepsilon$ and $\dd(K_0, K_1) > \varepsilon$.
\end{corollary}

\begin{proof}
	Take $z, z' \in K$ with $\dd(z,z') > 3 \varepsilon$; let $U_0, U_1$ be the open disks of radius $\varepsilon$ around them, and $K_0, K_1$ the closures of the components of $K \cap U_i$ containing $z, z'$. The distance property is clear from the triangle inequality, and the diameter property follows from Lemma \ref{lem:nadler}.
\end{proof}

\begin{proposition}\label{prop:Tight}
	For each nontrivial continuum $K \subset S$ we have $h_{\pm}(g,K) > 0$. In particular the essential factor is tight.
\end{proposition}

\begin{proof}
	We first prove that for every small $\varepsilon > 0$ there exists $n_{\varepsilon} > 0$ such that every continuum $K$ with $\mathrm{diam}(K) \geq \varepsilon$ admits $-n_\varepsilon \leq n_K \leq n_\varepsilon$ with $\mathrm{diam}(g^{n_K}(K)) > 3 \varepsilon$. If not, there is a sequence of continua $K_n$ with $\mathrm{diam}(K_n) \geq \varepsilon$ and $\mathrm{diam}(g^j(K_n)) \leq 3\varepsilon$ for $|j| \leq n$; a Hausdorff limit $K$ satisfies $\diam(K) \geq \varepsilon$ and $\mathscr{D}_g(K) \leq 3\varepsilon$, contradicting Lemma \ref{lem:essential_factor_ice}.
	
	Now iterate: given $K$ with $\diam(K) \geq \varepsilon$, take $n_K$ as above and split $g^{n_K}(K)$ into $K_0, K_1$ using Corollary \ref{cor:SplittingContinua}; recursively split each piece. After $j$ steps we obtain $2^j$ points of $K$ which are $(\pm j n_{\varepsilon}, \varepsilon)$-separated, so $s_{\pm j n_\varepsilon}(\varepsilon, K) \geq 2^j$ and $h_\pm(g,K) \geq \log 2 / n_\varepsilon > 0$.
\end{proof}
}

\section{Stable and unstable structure of the essential factor}\label{sec:stableg}

Throughout this section $g$ denotes the essential factor of some $f \in \Homeo_A(S)$ -- equivalently, by Remark \ref{rem:idempotent}, an infinitely continuum-wise expansive element of $\Homeo_A(S)$ -- with fixed lifts $\wt g, \wt A$ and Militon partitions $\eta^s, \eta^u$ as above.

\subsection{Shrinking stable and unstable continua}

The following compactness argument is the higher-genus version of
the shrinking argument in {\cite[Lemma 6.17]{garciatal24chaotic}}.

\begin{lemma}[\cite{garciatal24chaotic}]
\label{lem:stable_continua_shrink}
If ${\tK}\subset\widetilde S$ is a compact stable continuum of the
essential factor $g$, then
\[
 \lim_{n\to\infty}\diam(\widetilde g^{\,n}{\tK})=0.
\]
For a compact unstable continuum the same conclusion holds as
$n\to-\infty$.
\end{lemma}
\begin{proof}
Lemma~\ref{lem:stable_continua_equiv} gives
$M=\sup_{n\geq0}\diam(\widetilde g^{\,n}{\tK})<\infty$.
Suppose that $n_j\to\infty$ and
$\diam(\widetilde g^{\,n_j}{\tK})\geq\varepsilon>0$.
Choose deck transformations $\gamma_j$ so that
${\tK}_j=\gamma_j\widetilde g^{\,n_j}{\tK}$ meets a fixed compact
fundamental set. All ${\tK}_j$ lie in its compact $M$-neighbourhood.
A subsequence converges in the Hausdorff topology to a continuum
${\tK}_\infty$ of diameter at least $\varepsilon$.

For each fixed $m\in\mathbb Z$ and all sufficiently large $j$,
$n_j+m\geq0$ and
\[
 \widetilde g^{\,m}{\tK}_j
 =\bigl(\widetilde g^{\,m}\gamma_j\widetilde g^{-m}\bigr)
   \widetilde g^{\,n_j+m}{\tK}.
\]
The parenthesized map is a deck isometry. Thus
$\diam(\widetilde g^{\,m}{\tK}_j)\leq M$, and passage to the limit
gives $\diam(\widetilde g^{\,m}{\tK}_\infty)\leq M$ for every
$m\in\mathbb Z$. This contradicts infinite continuum-wise
expansiveness. Apply the same argument to $g^{-1}$ for the
unstable assertion.
\end{proof}

\subsection{Small stable and unstable continua}

Since $g$ is (infinitely) continuum-wise expansive, we may apply \cite[Theorem 1.6]{kato93fully}:

\begin{lemma}\label{lem:kato}
	For every $\varepsilon > 0$ there exists $\delta > 0$ such that for every $\wt z \in \tS$, the connected components containing $\wt z$ of the sets
	$$ \{\wt z' : \dd(\wt g^j \wt z', \wt g^j \wt z) \leq \varepsilon \ \forall j \geq 0\}, \qquad \{\wt z' : \dd(\wt g^j \wt z', \wt g^j \wt z) \leq \varepsilon \ \forall j \leq 0\} $$
	have diameter at least $\delta$. In particular, every point admits nontrivial stable and unstable continua of arbitrarily small {respectively future and past dynamical diameter.}
\end{lemma}

\begin{corollary}\label{cor:emptyinterior}
	Every stable {pseudo-leaf} $\eta^s(\xi)$ and every unstable {pseudo-leaf} $\eta^u(\xi')$ of $g$ has empty interior. Consequently, stable and unstable continua of $g$ are filled continua with empty interior.
\end{corollary}

\begin{proof}
	Suppose a ball $B$ were contained in a stable {pseudo-leaf}. By Lemma \ref{lem:kato}, $B$ contains a nontrivial unstable continuum $\wt K^u$; then $\wt K^u$ would be simultaneously a stable and an unstable continuum, hence a point by Corollary \ref{cor:trivialintersections}, a contradiction. Filledness: a bounded complementary component of a stable continuum would be an open set whose forward iterates have bounded diameter (they are enclosed by the iterates of the continuum), so the union of its closure with the continuum is again forward-bounded, and its interior points would give a ball inside a stable {pseudo-leaf} by Lemma \ref{lem:stable_continua_equiv}, contradicting the first part.
\end{proof}

\subsection{\texorpdfstring{Density of stable and unstable {pseudo-leaves}}{Density of stable and unstable pseudo-leaves}}

The next proposition transfers minimality of the model foliations
to the {pseudo-leaves} of the essential factor. It supplies
dense stable and unstable continua for the local return construction. For $z \in S$, denote by $\mathcal{W}^s(z) \subset S$ the projection $\wt\pi(\eta^s(\xi^s(\wt z)))$ of the stable {pseudo-leaf} of any lift, and similarly $\mathcal{W}^u(z)$.

\begin{proposition}\label{thm:density}
	Let $g$ be as above and let $z \in S$ be a point admitting a lift $\wt z$ whose stable class is the class of a leaf: $\xi^s(\wt z) = [M]$ for some $M \in \mathcal{L}^s$. Then $\mathcal{W}^s(z)$ is dense in $S$. The symmetric statement holds for unstable {pseudo-leaves}. In particular, the projected {pseudo-leaves} of every {genuine} model leaf are dense.
\end{proposition}

\begin{proof}
	Let $U \subset S$ be open, and let $\wt U$ be a lift. By Lemma \ref{lem:kato}, take a nontrivial unstable continuum $\wt K^u \subset \wt U$. Since $g$ is infinitely continuum-wise expansive, $\wt K^u$ is not contained in a single stable {pseudo-leaf} (Corollary \ref{cor:trivialintersections}), so there exist $\wt z', \wt z'' \in \wt K^u$ with $\xi' := \xi^s(\wt z') \neq \xi'' := \xi^s(\wt z'')$. By Lemma \ref{lem:separating_translate} applied to the leaf $M$ and the pair $\xi', \xi''$, there is $\gamma \in \Gamma$ such that $\gamma M$ separates $\xi'$ from $\xi''$. By Lemma \ref{lem:intermediate} applied to the continuum $\wt K^u$,
	$$ \wt K^u \cap \eta^s([\gamma M]) \neq \varnothing.$$
	Since $\eta^s([\gamma M]) = \gamma\, \eta^s([M])$ by (D2 -- Equivariance), and $\wt K^u \subset \wt U$, projecting to $S$ gives $U \cap \mathcal{W}^s(z) \neq \varnothing$.
\end{proof}

\section{Local rotational horseshoes and mixing}\label{sec:horseshoes}

{We first construct local heteroclinic pseudo-rectangles. A Markovian-type argument gives local rotational horseshoes. }We then prove eventual intersections
of stable and unstable continua, which imply mixing and the density
of topologically transverse intersections between periodic points.

Throughout this section $g$ is an infinitely continuum-wise expansive
homeomorphism isotopic to a closed pseudo-Anosov map $A$. Choose a
nonsingular periodic point $p_A$ of $A$, a lift $\wt p_A$, and a
positive integer $n$ such that a lift
$A'=\gamma_0\wt A^n$ fixes $\wt p_A$ and each of the four ends
of its invariant leaves. Increase $n$ if necessary so that $A^n$
preserves orientation. Set
\[
 G=\gamma_0\wt g^n,\qquad
 \xi^s_0=[L^s(\wt p_A)],\qquad
 \xi^u_0=[L^u(\wt p_A)].
\]
{Both classes are represented by regular leaves, and thus their pseudo-leaves
are nonempty in $\tS$} by Theorem~\ref{thm:militon_partition}, and
$G$ preserves them by Lemma~\ref{lem:induced_actions}. Write
\[
 \mathcal W^s_0=\wt\pi(\eta^s(\xi^s_0)),\qquad
 \mathcal W^u_0=\wt\pi(\eta^u(\xi^u_0)).
\]
These two projected {pseudo-leaves} are dense by Proposition~\ref{thm:density}.
Let $e^s_1,e^s_2$ be the ends of $L^s(\wt p_A)$ and
$e^u_1,e^u_2$ those of $L^u(\wt p_A)$. The stable ends are
repelling and the unstable ends attracting for the common boundary
action of $G$ and $A'$; their cyclic order alternates. Put
$a=\lambda^n\in(0,1)$. The induced action of $A'$ contracts the
unstable {class} tree by $a$ and expands the stable {class} tree by
$a^{-1}$.

\subsection{Pseudo-rectangles and prime ends}

A \emph{marked rectangle} is a closed Jordan disk whose boundary is
partitioned into four arcs in cyclic order. Opposite arcs form its
two pairs of sides, {which we call \emph{horizontal} and \emph{vertical}}. A homeomorphism from $[0,1]^2$ preserving these
markings is called a rectangle chart.
\label{subsec:primeends}

We use Carath\'eodory's prime end compactification as in
\cite[Sections 2.6, 8]{garciatal24chaotic}, including crosscuts,
fundamental chains, impressions and accessible prime ends. The
nearby-rectangle lemma will be
applied to the pseudo-rectangles constructed in
Proposition~\ref{prop:HPRdense}.

\begin{definition}[Generalized rectangle, {\cite[Definition 8.8]{garciatal24chaotic}}]\label{def:GeneralizedRectangle}
	An inessential continuum $G \subset S$ is a \emph{generalized rectangle} if
	\begin{enumerate}
		\item $G$ is the closure of an open disk $D$;
		\item (Sides) $\partial G = K^{\textnormal{h}}_1 \cup K^{\textnormal{h}}_2 \cup K^{\textnormal{v}}_1 \cup K^{\textnormal{v}}_2$, each a continuum, with $K^{\textnormal{h}}_1 \cap K^{\textnormal{h}}_2 = \varnothing$, $K^{\textnormal{v}}_1 \cap K^{\textnormal{v}}_2 = \varnothing$;
		\item the associated prime end sides define a rectangle structure on the prime end compactification $\breve{G}$, as in \cite[Lemma 8.6]{garciatal24chaotic}.
	\end{enumerate}
	We call $K^{\textnormal{h}}_i$ the horizontal and $K^{\textnormal{v}}_i$ the vertical sides. In item (3), the \emph{prime end side associated to} $K^{\sigma}_i$ is the closure, in the prime-end circle, of the accessible prime ends admitting an access ray whose image lands in $K^{\sigma}_i$, as in \cite[Definition 8.4]{garciatal24chaotic}; condition (3) asks that these four sets be arcs of the circle of prime ends, arranged in the cyclic order $\textnormal{h}_1, \textnormal{v}_1, \textnormal{h}_2, \textnormal{v}_2$ and meeting only at their endpoints, so that $\breve{G}$ is a closed disk whose boundary is subdivided as the boundary of $[0,1]^2$; see \cite[Definition 8.8]{garciatal24chaotic}.
\end{definition}

{The following result follows from \cite[Lemma 8.9]{garciatal24chaotic}, where it was proved for generalized rectangles.

\begin{lemma}
	\label{lem:compact_crossing_duality}
	Let $K$ be a compact subset of $[0,1]^2$.  If $K$ meets every path
	in the square from the left side to the right side, then $K$
	contains a continuum joining the bottom side to the top side.
\end{lemma}
}
\begin{definition}[{Heteroclinic pseudo-rectangle}, \cite{garciatal24chaotic}]\label{defi:ChaoticRectangle}
	An inessential continuum $P \subset S$ is a \emph{heteroclinic pseudo-rectangle} (HPR) for $g$, {subordinated to classes $\alpha_1\ne\alpha_2$ in $\slea$ and $\beta_1\ne\beta_2$ in $\ulea$, if the following hold:}
	\begin{enumerate}
		\item $P$ is the closure of an open disk $D$;
		\item for any $z \in \partial P$, any neighbourhood of $z$ intersects the fully essential component $P_\infty$ of the complement of $P$;
		\item $\partial P = K^s_1 \cup K^s_2 \cup K^u_1 \cup K^u_2$, {and there is a lift $\wt P$ such that the corresponding $\wt K^s_i\subset\eta^s(\alpha_i)$ are stable continua and the corresponding $\wt K^u_i\subset\eta^u(\beta_i)$ are unstable continua;}
		\item $K^s_i \cap K^u_j \neq \varnothing$ for every $i,j \in \{1,2\}$.
	\end{enumerate}
	A lift $\wt P \subset \tS$ of an HPR is again called an HPR; its sides are the corresponding lifts.
\end{definition}

{
\begin{remark}\label{rem:HPRbasic}
	Since $\alpha_1 \neq \alpha_2$, the corresponding {pseudo-leaves are disjoint for the lift $\wt P$}, so $K^s_1 \cap K^s_2 = \varnothing$, and similarly for the unstable sides; the sides intersect pairwise in totally disconnected sets by Corollary \ref{cor:trivialintersections}. Every HPR, and any of its lifts, is a generalized rectangle by {\cite[Lemma 8.6]{garciatal24chaotic}.}
\end{remark}
}

We use the following approximation result from \cite[Lemma 8.7]{garciatal24chaotic}.

\begin{lemma}\label{lem:RectangleNearPseudoRectangle}
	Let $G$ be an HPR as in Definition~\ref{defi:ChaoticRectangle},
with sides $K^{\textnormal{h}}_1, K^{\textnormal{h}}_2,
K^{\textnormal{v}}_1, K^{\textnormal{v}}_2$. For every $\varepsilon>0$
there exists a rectangle $R\subset G$ {with respective sides $R^{\textnormal{h}}_1, R^{\textnormal{h}}_2,
R^{\textnormal{v}}_1, R^{\textnormal{v}}_2$} such that
	$$ \dd(z, K^{\sigma}_i) < \varepsilon \text{ for every } z \in R^{\sigma}_i, \ \sigma \in \{\textnormal{h}, \textnormal{v}\}, \ i = 1,2. $$
\end{lemma}

\subsection{Crossing rectangles}

{Throughout this subsection, $\Sigma$ is any surface. A \emph{horizontal arc} of a marked rectangle is contained in it and has its endpoints
contained in opposite vertical sides of the rectangle; similarly a \emph{vertical arc} has its
endpoints contained in opposite horizontal sides of the rectangle. Homeomorphisms transport the markings.

We now introduce a topological ordered-crossing notion {for} marked rectangles. It was developed {with contributions} from Favio Pirán. 
}

{
\begin{definition}[Crossing rectangle]\label{def:crossing_rectangle}
	Let $A, B \subset \Sigma$ be marked rectangles, and $m \geq 1$ an integer. We will say that \emph{$A$ $m$-crosses $B$} and {write} $A \dagger_m B$, if any horizontal arc of $A$ intersects any vertical arc of $B$ in at least $m$ distinct points.  
		
	When $m=1$, we simplify notation by saying that $A$ \emph{crosses} $B$ and {writing} $A \dagger B$.
\end{definition}
}

{
\begin{lemma}\label{lem:crossing_rectangle_equivalence}
	Having $A \dagger_m B$ is equivalent to {either of the following two assertions:}
	\begin{itemize}
		\item Any horizontal arc of $A$ contains $m$ disjoint horizontal arcs of $B${.}
		\item Any vertical arc of $B$ contains $m$ disjoint vertical arcs of $A${.}
	\end{itemize}
\end{lemma}

\begin{proof}
	Let us prove the equivalence with the former condition{; the} proof for the latter is symmetric. 
	
	For the converse implication: for any horizontal arc of $A$, each of its $m$ disjoint horizontal arcs of $B$ intersects every vertical arc of $B$ by {Lemma}~\ref{lem:crossedcontinua}. 
	
	For the direct implication: suppose by contradiction that there exists a horizontal arc $a_h$ of $A$ containing at most $m-1$ disjoint horizontal arcs of $B$. Divide the subarcs of $a_h \cap B$ into {two classes. First we consider the collection of the minimal subarcs that are horizontal subarcs of $B$. These are the crossing subarcs. Any two of these arcs can meet at most in an endpoint (in which case we call them consecutive arcs) and by the pigeonhole principle there are at most $2m-2$ of these. Then let us consider all other maximal subarcs of  $a_h \cap B$ that do not meet the interior of a crossing subarc. These are noncrossing horizontally in $B$. By Lemma \ref{lem:compact_crossing_duality} there exists a vertical continuum that is disjoint from all the noncrossing subarcs. One can find a vertical arc $b_v$ of $B$, arbitrarily close to this continuum which is also disjoint from the noncrossing subarcs. Furthermore, $b_v$ can be modified such that it only intersects crossing subarcs at a single point, and if $(a'_{i_j})_{1\le j\le n}$ is a maximal sequence of crossing arcs such that $a'_{i_j}$ and $a'_{i_j+1}$ are consecutive, then $b_v$ intersects this sequence only at points of the form $a'_{i_{2s-1}}\cap a'_{i_{2s}}$.} Thus $a_h \cap b_v$ has cardinality at most $m-1$, which is a contradiction. 
\end{proof}
}

This has the following immediate consequences.

\begin{corollary}\label{cor:crossing_basic_properties}
	$m$-crossing between marked rectangles has the following properties: 
	\begin{itemize}
		\item Equivariance. If $A \dagger_m B$ and $h$ is a homeomorphism of $\Sigma$, then $h(A) \dagger_m h(B)${.}
		\item Transitivity. If $A\dagger_m B, B\dagger_m C$, then $A\dagger_m C$ ---furthermore, $A \dagger_{m^2} C$---{.}
	\end{itemize}
	{Moreover}, 1-crossing is reflexive ---$A\dagger A$ for every $A$--- and antisymmetric: whenever $A \dagger B$ and $B\dagger A$, then $A =B$. 
\end{corollary}

In particular, 1-crossing induces a partial order {on} the family of marked rectangles of $\Sigma$. 

\medskip

The following is a sufficient condition for an ordered pair of rectangles to have a 1-crossing.

	\begin{lemma}\label{lem:crossing_configuration}
		Let $\ell_1,\ldots,\ell_4\subset\mathbb R^2$ be disjoint proper
		lines with chosen open topological half-planes $H_1,\ldots,H_4$
		which are pairwise disjoint and occur in this cyclic order.
		Let $R_1,R_2$ be marked rectangles such that:
		\begin{itemize}
			\item the vertical sides of $R_1$ lie respectively in $H_1,H_3$,
			and $R_1\cap(H_2\cup H_4)=\varnothing$;
			\item the horizontal sides of $R_2$ lie respectively in $H_2,H_4$,
			and $R_2\cap(H_1\cup H_3)=\varnothing$.
		\end{itemize}
		Then, $R_1 \dagger R_2$.
	\end{lemma}
	
	\begin{proof}
		Any vertical arc of ${R_2}$ contains a subarc from $\ell_2$ to $\ell_4$: it then separates the vertical sides of ${R_1}$ and thus intersects every horizontal arc of ${R_1}$.  
	\end{proof}

We now introduce dynamics into the crossings.

\begin{lemma}\label{lem:compact_witness_crossing}
	Let $h \in \Homeo(\Sigma)$ and let $R$ be a marked rectangle such that $h(R) \dagger R$. Then there exists a compact set $K \subset R$, such that every horizontal arc of $R$ contains a subarc in $K$ whose image is a horizontal arc of $R$.
\end{lemma}

\begin{proof}
	It is enough to take $K = h^{-1}(R) \cap R$ and use the first equivalence from Lemma \ref{lem:crossing_rectangle_equivalence}.
\end{proof}

We call such a compact set an \emph{$h$-witness} of $R$. From the duality of the crossing definition, we have that $h(K)$ is an $h^{-1}_\ast$-witness of $R$, meaning that every vertical arc of $R$ contains a subarc in $h(K)$ whose image by $h^{-1}$ is vertical.

\begin{lemma}\label{lem:witness_continuum}
	If $h(R) \dagger R$, an $h$-witness contains a vertical continuum of $R$. Similarly, an $h^{-1}_\ast$-witness contains a horizontal continuum of $R$.  
\end{lemma}

\begin{proof}
	By definition, an $h$-witness must intersect every horizontal arc of $R$. The assertion follows from Lemma \ref{lem:compact_crossing_duality}.
\end{proof}

Taking successive restrictions of horizontal arcs gives us the following: 

\begin{lemma}\label{lem:witness_intersection}
	If $h_1(R)\dagger R$ and $h_2(R)\dagger R$, with $K_1, K_2$ respective witnesses, then $K' = h_2^{-1}(K_1) \cap R$ is an $h_1 \circ h_2$-witness, and in particular contains a vertical continuum.
	
	In particular, if $h(R)\dagger R$ and $K_1, K_2$ are disjoint $h$-witnesses, then $h^{-1}(K_1) \cap K_2$ is an $h^2$-witness.
\end{lemma}

\begin{proof}
	A horizontal arc $a$ of $R$ has a subarc $a'$ in $K_2$ whose image is horizontal, thus $h(a')$ has a subarc $a''$ in $K_1$ whose image is horizontal. Then $h_2^{-1}(a'')$ is a subarc of $a$ in $K'$ whose image under $h_1 \circ h_2$ is horizontal.    
\end{proof}

The next result allows to obtain realization of periodic points for rotational horseshoes.

\begin{lemma}\label{lem:crossing_fixed_point}
	Let $R\subset \Sigma$ be a marked rectangle, $h \in \Homeo(\Sigma)$, suppose that $h(R) \dagger R$, and let $K$ be an $h$-witness. Then $\mathrm{Fix}(h|_K) \ne \varnothing$. 
\end{lemma}
\begin{proof}
The pair $(K,h)$ satisfies the hypotheses of \cite[Theorem~3.9]{papinizanolin04fixed}, and thus $h$ has a fixed point in $K\subset R$.
\end{proof}

{
\begin{definition}
[Rotational topological horseshoe]
\label{def:rotational_horseshoe}
Let $H$ be a lift of a surface homeomorphism $h$.
A marked rectangle $\wt R\subset\tS$, projecting injectively to $S$,
defines a \emph{rotational (topological) horseshoe} with pairwise distinct deck
transformations $T_1,\ldots,T_m$, $m\geq2$, if
\[
H(\wt R)\dagger T_l \wt R, \text{ for every }1 \leq l \leq m.
\]
\end{definition}
}

Note that Markovian rotational horseshoes from \cite{guiheneuf26theory} satisfy our definition, and {that these horseshoes are also topological horseshoes as defined in} \cite{lecalveztal22topological}.

{
\begin{proposition}\label{prop:realization_rotational_horseshoe}
	Suppose that $\wt R$ projects injectively to $S$ and defines a rotational topological horseshoe for $H$ with associated deck transformations $T_1,\ldots,T_m$, $m\geq2$. Then
	$h$ has a compact invariant set $\Lambda\subset\wt\pi(\wt R)$
	and a semiconjugacy
	\[
	\theta:\Lambda\longrightarrow\{1,\ldots,m\}^{\mathbb Z},
	\qquad \theta h=\sigma\theta.
	\]
	Every periodic itinerary has a periodic realization.
\end{proposition}

{
\begin{proof}
For $1\leq l\leq m$, set $H_l=T_l^{-1}H$. Since $H_l(\widetilde R)\dagger\widetilde R$, Lemma~\ref{lem:compact_witness_crossing} gives a vertical witness
\[
 \widetilde K_l\subset\widetilde R\cap H^{-1}(T_l\widetilde R)
\]
for $H_l$.
The witnesses are pairwise disjoint because the translates $T_l\widetilde R$ are pairwise disjoint. Their projections $K_l=\widetilde\pi(\widetilde K_l)$ are also disjoint, {and they are $m$ disjoint $h$-witnesses ---in particular, $h(R) \dagger_m R$---. Take a two-sided finite itinerary $w=(l_{-n}, \dots, l_n)$ and define
\[
 K=\bigcup_{l=1}^m K_l,\qquad
 \Lambda=\bigcap_{r\in\mathbb Z}h^{-r}(K), \qquad
 K_{w^-}=
 \bigcap_{r=0}^{n}h^{-r}(K_{l_r}), \qquad
 K_{w^+}=
 \bigcap_{r=1}^{n}h^{r}(K_{l_{-r}}), \qquad 
\]

By Lemma \ref{lem:witness_intersection}, we have that $K_{w^+}$ is an $h^{n+1}$-witness which contains a vertical continuum of $R$, and $K_{w^+}$ is an $h^{-n}_{\ast}$-witness which contains a horizontal continuum of $R$, thus $K_w = K_{w^-} \cap K_{w^+} \ne \varnothing$.} Applying this argument to finite blocks of a two-sided itinerary and taking nested intersections of the corresponding compact cylinders gives a point of $\Lambda$ with that itinerary. Hence
\[
 \theta(x)_r=l\quad\Longleftrightarrow\quad h^r(x)\in K_l
\]
defines a surjection onto the full two-sided shift. It is continuous because the finitely many compact sets $K_l$ are disjoint, and $\theta h=\sigma\theta$.

For a finite word $w=(l_1,\ldots,l_k)$, let $P_0=\mathrm{Id}$,
$P_j=H_{l_j}\circ\cdots\circ H_{l_1}$, and restrict the composition to
\[
 \widetilde K_w=
 \bigcap_{j=0}^{k-1}P_j^{-1}(\widetilde K_{l_{j+1}}).
\]
This restriction enforces every intermediate symbol.
By {Lemma \ref{lem:witness_intersection}}, $\wt K_w$ is a $P_k$-witness, and thus {Lemma \ref{lem:crossing_fixed_point}} gives a fixed point $x\in\widetilde K_w$ of $P_k$ {following the prescribed deck transformations}. Its projection is fixed by $h^k$.
\end{proof}
}
}

{
We give one more notion of topological horseshoe with realization for the sake of {completeness}.

\begin{definition}[Inessential topological horseshoe]\label{def:inessential_topological_horseshoe}
	Let $H$ be a lift of a surface homeomorphism $h$.
	A marked rectangle $\wt R\subset\tS$, projecting injectively to $S$,
	defines a \emph{inessential (topological) horseshoe} if $H(\wt R) \dagger_m \wt R$, $m\ge 2$. 
\end{definition} 
}

{
The following assertion links this notion with the one from \cite{kennedyyorke01topological}.

\begin{lemma}\label{lem:inessential_vs_crossing_horseshoe}
	If $h(R) \dagger_m R$, then $R$ contains a topological horseshoe in the sense of \cite{kennedyyorke01topological}. In particular, $R$ contains $m$ disjoint $h$-witnesses.
\end{lemma}

\begin{proof}
	Define the \textit{ends} $\mathrm{end}_0, \mathrm{end}_1$ from \cite{kennedyyorke01topological} as the vertical sides of $R$. By Lemma \ref{lem:crossing_rectangle_equivalence} every horizontal arc of $h(R)$ contains $m$ disjoint horizontal subarcs of $R$; the whole construction from \cite[Theorem 1]{kennedyyorke01topological} can be applied verbatim, replacing \textit{connections} by horizontal arcs. This construction gives $m$ disjoint compact \textit{symbol sets} $K_j$: each horizontal arc of $R$ contains a arc in $K_j$ whose image is horizontal, thus $K_j$ is an $h$-witness for every $j$. To check that same assertion for horizontal continua and obtain a topological horseshoe in the sense of \cite{kennedyyorke01topological}, approximate a continuum by sufficiently close curves in the Hausdorff topology. The existence of a subarc ---respectively subcontinuum--- in each $K_j$ whose image is horizontal passes directly to the limit.
	 
\end{proof}
}

{
\begin{proposition}\label{prop:inessential_topological_horseshoe}
		Suppose that $\wt R$ projects injectively to $S$ and defines an inessential topological horseshoe for $H$.  
		
		Then
		$h$ has a compact invariant set $\Lambda\subset\wt\pi(\wt R)$
		and a semiconjugacy
		\[
		\theta:\Lambda\longrightarrow\{1,\ldots,m\}^{\mathbb Z},
		\qquad \theta h=\sigma\theta.
		\]
		Every periodic itinerary has a periodic realization.
\end{proposition}

\begin{proof}
{
By Lemma \ref{lem:inessential_vs_crossing_horseshoe}, we can applt \cite[Theorem 1]{kennedyyorke01topological} for $\wt R$. Its symbol-set construction provides pairwise disjoint compact sets $\widetilde K_1,\ldots,\widetilde K_m\subset\widetilde R$ such that every connection contains a preconnection in each $\widetilde K_l$. Intersecting these sets with $H^{-1}(\widetilde R)$ if necessary, we may also assume $H(\widetilde K_l)\subset\widetilde R$. On a horizontal arc each such preconnection is a subarc, so every $\wt K_l$ is an $H$-witness.}

Set $\widetilde K=\bigcup_l\widetilde K_l$ and $\widetilde\Lambda=\bigcap_{r\in\mathbb Z}H^{-r}(\widetilde K)$. The itinerary construction in the preceding proof gives a continuous surjection from $\widetilde\Lambda$ onto the full $m$-shift. For a word $w=(l_1,\ldots,l_k)$, the compact set
\[
 \widetilde K_w=\bigcap_{j=0}^{k-1}H^{-j}(\widetilde K_{l_{j+1}})
\]
is an $H^k$-witness. Then, {Lemma \ref{lem:crossing_fixed_point}} gives a point of $\widetilde K_w$ fixed by $H^k$, realizing the prescribed periodic itinerary. Finally, $\widetilde\pi$ is injective on $\widetilde R$, so projection gives the required set $\Lambda$, semiconjugacy, and periodic realizations for $h$.
\end{proof}
}

\subsection{Local construction of heteroclinic pseudo-rectangles}

All continua and {pseudo-leaves} in this subsection are considered in one
fixed lift of a disk of $S$.  In particular, distinctness of two side
classes always means distinctness upstairs.  This convention matters:
two distinct translates of a {pseudo-leaf} can have the same projection.

\begin{lemma}\label{lem:local_crossing_supply}
	Inside sufficiently small concentric lifted disks ${\wt U}'\subset {\wt U}$ there
	are infinitely many pairwise disjoint stable continua crossing the
	closed annulus ${\overline{\wt U}}\setminus {\wt U}'$, each in a different translate
	of the {pseudo-leaf} $\eta^s(\xi^s_0)$. Their diameters and their
	forward dynamical diameters have common upper bounds. The symmetric
	unstable statement holds.
\end{lemma}
\begin{proof}
	The union of the translates of this {pseudo-leaf} is dense by
	Proposition~\ref{thm:density}. A finite union of distinct {pseudo-leaves} is
	closed with empty interior, so infinitely many distinct translates
	meet ${\wt U}'$. Through a chosen point of each such {pseudo-leaf} there is an
	unbounded connected component, by boundary bumping after
	Theorem~\ref{thm:militon_partition}. Lemma~\ref{lem:large_compact_pieces}
	and boundary bumping give a subcontinuum crossing the annulus, which
	may be taken minimal for inclusion. It is contained in ${\overline{\wt U}}$,
	so the diameters are uniformly bounded, and
	Lemma~\ref{lem:stable_continua_equiv} bounds all forward iterates.
	{Distinct classes give disjoint pseudo-leaves, hence disjoint continua.}
\end{proof}

The next lemma controls the {size of the} unstable continua lying {sufficiently close to a stable continuum. Its proof comes from \cite[Proposition 8.11, Step 4]{garciatal24chaotic}}.

\begin{lemma}
\label{lem:HPR_thin_strip}
Let ${\wt C}$ be a compact stable continuum in a compact lifted disk ${\wt B}$.
For every $r>0$ there is a neighbourhood $O$ of ${\wt C}$ such that every
unstable continuum contained in $O\cap {\wt B}$ has diameter less than $r$.
The symmetric assertion also holds.
\end{lemma}

\begin{proof}
Otherwise there are unstable continua ${\wt J}_n\subset {\wt B}$, with
$\operatorname{diam}({\wt J}_n)\geq r$ and
$\sup_{x\in {\wt J}_n}d(x,{\wt C})\to0$.  A subsequence has a Hausdorff limit
${\wt J}\subset {\wt C}$ with diameter at least $r$.  This limit is an unstable
continuum, by the Hausdorff-limit property for unstable continua.
Equivalently, continuity of the unstable class map on ${\wt B}$ shows that
it is constant on ${\wt J}$: it is constant on every ${\wt J}_n$, and this property
passes to the Hausdorff limit.  The continuum ${\wt J}$ is therefore both
stable and unstable, contradicting
Corollary~\ref{cor:trivialintersections}.
\end{proof}

{
\begin{proposition}\label{prop:HPRdense}
	Every nonempty open set $U\subset S$ contains an HPR.  In a lift
	projecting injectively to $U$, its stable sides {have the distinct classes} $\alpha_1 = \gamma_1\xi^s_0, \alpha_2 = \gamma_2\xi^s_0$ and its unstable sides {have the distinct classes} $\beta_1 = \delta_1\xi^u_0, \beta_2 = \delta_2\xi^u_0$, where
	$\gamma_i,\delta_i\in\Gamma$.
\end{proposition}

\begin{proof}
	We follow the construction of \cite[Proposition 8.11]{garciatal24chaotic}, indicating the substitutions; the construction is topological and rests on four inputs, all available here:
	\begin{enumerate}[label=(\roman*)]
		\item \emph{density}: the projected {pseudo-leaves $\mathcal W^s_0$ and $\mathcal W^u_0$} are dense (Proposition \ref{thm:density}, applicable since $\xi^s_0, \xi^u_0$ are classes of leaves), {and distinct translates of $\eta^s(\xi^s_0)$ are disjoint}, {so by Lemma \ref{lem:local_crossing_supply}} inside any small annulus we find countably many pairwise disjoint stable continua lying in {pseudo-leaves with pairwise distinct classes} $\gamma \xi^s_0$, and similarly for unstable continua;
		\item \emph{limits}: Hausdorff limits of stable continua with uniformly bounded forward iterates are stable continua (Corollary \ref{cor:hausdorff_limit_stable});
		\item \emph{shape}: stable and unstable continua are filled and have empty interior (Corollary \ref{cor:emptyinterior});
		\item \emph{transversality}: the intersection of a stable and an unstable continuum is totally disconnected (Corollary \ref{cor:trivialintersections}).
	\end{enumerate}
	With these inputs, the {six} steps of \cite[Proposition 8.11]{garciatal24chaotic} go through unchanged. In outline: inside a lifted annulus $\wt U \setminus \wt U' \subset \wt U \subset \tS$, one takes a sequence of pairwise disjoint minimal crossing stable continua $\wt K^s_n$ {with distinct stable classes} from (i), extracts a Hausdorff limit $\wt K^s$ which is a stable continuum by (ii) and can be assumed disjoint from all $\wt K^s_n$: indeed $\wt K^s$ {lies in one stable pseudo-leaf}, because a continuum meeting two distinct stable {pseudo-leaves} has forward iterates of diameter tending to infinity (Corollary \ref{cor:two_classes_unbounded}), contradicting that $\wt K^s$ is a stable continuum; by (iii) the limit is inessential in the annulus and one obtains the local disk picture of \cite[Figure 17]{garciatal24chaotic}; repeating the construction with unstable continua inside that disk, and using (iv) to control the corners, one {finds infinitely many {transverse intersections} ---using Lemma \ref{lem:HPR_thin_strip}---} and assembles the boundary of the pseudo-rectangle exactly as in \cite[Steps 3--6]{garciatal24chaotic}. In the torus case the disjointness of the stable sets used in (i) came from distinct rotation vectors; here it is furnished directly by the distinctness of the classes, which simplifies the argument.
\end{proof}
}

\subsection{Compact sets and the boundary attractors}

\begin{lemma}\label{lem:compact_confinement}

For every compact {$\tK\subset\tS$} and every neighbourhood $V$ of
$\overline\eta^u(\xi^u_0)$ in $\overline S$, one has {$G^j\tK\subset V$}
for all sufficiently large $j$.
\end{lemma}
\begin{proof}
Continuity of the class map gives
{$D_{\tK}=\sup_{x\in\tK}\dd^s(\xi^u(x),\xi^u_0)<\infty$}. By equivariance
and contraction,
\[
{ \sup_{x\in\tK}\dd^s(\xi^u(G^jx),\xi^u_0)=a^jD_{\tK}\longrightarrow0.}
\]
Apply property {(T4 - Continuity)} at the nonempty invariant leaf class $\xi^u_0$.
It controls the entire {compactified pseudo-leaf} of every image point,
uniformly in {$x\in\tK$}, and gives the conclusion.
\end{proof}

\begin{lemma}\label{lem:side_escape}
Let ${\wt E}\subset\tS$ be compact and suppose all its {stable classes} lie
on one side of the regular {leaf class} $\xi^s_0$, at positive distance
from it. Label the corresponding side of
$\overline S\setminus\overline\eta^s(\xi^s_0)$ by the unique
attracting end $e^u_b$ it contains. Then $G^j{\wt E}$ converges to
$\{e^u_b\}$ in the Hausdorff topology of $\overline S$.
\end{lemma}
\begin{proof}
Compactness gives
$\delta_{{\wt E}}=\inf_{x\in {\wt E}}\dd^u(\xi^s(x),\xi^s_0)>0$.
Choose $w_0\in\eta^s(\xi^s_0)$. Lemma~\ref{lem:point_class_distance}
gives, uniformly in $x\in {\wt E}$,
\[
 \dd(G^jx,w_0)\geq a^{-j}\delta_{{\wt E}}-2C\longrightarrow\infty.
\]
Thus every accumulation point is on the boundary. By
Lemma~\ref{lem:compact_confinement}, it belongs to
$\overline\eta^u(\xi^u_0)$ and hence is either $e^u_1$ or $e^u_2$.
The side of $\overline\eta^s(\xi^s_0)$ is preserved by $G$, since
the prongs are fixed. The other attracting end is not in its closure.
Only $e^u_b$ remains, proving uniform convergence.
Lemma~\ref{lem:class_sides} gives the side assignment, including
{completion classes}. The compactified separation is within one
foliation throughout.
\end{proof}

\begin{lemma}\label{lem:coarse_leaf_tracing}

Let ${\wt K} \subset\tS$ be a continuum with points whose stable
{classes} lie on opposite sides of $\xi^s_0$. Set
\[
 D_{\wt K}=\sup_{x\in {\wt K}}\dd^s(\xi^u(x),\xi^u_0),
 \qquad L^u_0=L^u(\wt p_A).
\]
For every $v\in L^u_0$, there is $J(v,{\wt K})$ such that for every
$j\geq J(v,{\wt K})$ one can choose {$x_{j,v}\in G^j\tK$} with
\begin{equation}\label{eq:coarse_leaf_tracing}
 \dd(x_{j,v},v)\leq 2C+a^jD_{\wt K}{.}
\end{equation}
In particular the eventual bound $2C+1$ is independent of $v$.
The assertion has a time-reversed stable-leaf version.
\end{lemma}
\begin{proof}
Choose $x_-,x_+\in {\wt K}$ on the two stable sides.
Lemma~\ref{lem:side_escape} gives $G^jx_\pm\to e^u_\pm$.
For $v\in L^u_0$ let $M_v=L^s(v)$ and $\beta_v=[M_v]$.
The stable leaf $M_v$ crosses the regular unstable leaf $L^u_0$
at $v$. Thus two ends of $M_v$ interleave $e^u_-,e^u_+$;
this also holds if $M_v$ is singular, by choosing the two
directions at its nonsingular crossing with $L^u_0$.
The compactified continuum $\overline\eta^s(\beta_v)$ avoids
$e^u_-,e^u_+$ and separates them, by
Lemma~\ref{lem:crossedcontinua}. Consequently the connected set
$G^j{\wt K}$ meets $\eta^s(\beta_v)$ for all sufficiently large $j$.
Choose $x_{j,v}$ in this intersection.

The estimate {from Lemma \ref{lem:point_class_distance}} gives
$\dd^u(x_{j,v},v)\leq C$, since {its stable class is $[M_v]$}. Also, {by Lemma \ref{lem:compact_confinement}}, every point of $G^j{\wt K}$ has unstable {class} at distance at most $a^jD_{\wt K}$
from $\xi^u_0$. Hence
$\dd^s(x_{j,v},v)\leq C+a^jD_{\wt K}$. This completes the proof.
\end{proof}

\begin{remark}\label{rem:tracing_without_expansivity}
Lemmas~\ref{lem:compact_confinement}, \ref{lem:side_escape}, and
\ref{lem:coarse_leaf_tracing} apply to every homeomorphism isotopic
to a closed pseudo-Anosov map. Their proofs use the invariant
ordinary reference classes and their induced tree dynamics. This
observation will be used before taking a quotient in
Section~\ref{sec:PARF}.
\end{remark}

\begin{lemma}\label{lem:crossing_band}
Let $\beta_1,\beta_2$ be regular unstable {leaf classes} separated by
$\xi^u_0$, and let ${\wt C}\subset\tS$ be a continuum meeting both
$\eta^u(\beta_1)$ and $\eta^u(\beta_2)$. The continuum
\[
 {\wt B}=\overline\eta^u(\beta_1)\cup {\wt C}\cup\overline\eta^u(\beta_2)
\]
separates $e^u_1$ from $e^u_2$ in $\overline S$ and avoids these points.
Consequently any continuum joining sufficiently small neighbourhoods
of the two ends, while avoiding the two unstable {pseudo-leaves}, meets ${\wt C}$.
\end{lemma}
\begin{proof}
Choose an end $b_i$ of the leaf defining $\beta_i$. By same-foliation
separation the $b_i$ lie in opposite components of
$\partial\tS\setminus\{e^u_1,e^u_2\}$, so the pairs
$\{b_1,b_2\}$ and $\{e^u_1,e^u_2\}$ interleave. The set ${\wt B}$ is
connected, compact, and avoids $e^u_1,e^u_2$, since distinct unstable
leaves have disjoint end sets and ${\wt C}$ is compact in $\tS$.
If the two ends belonged to one complementary component, an arc in
that open component would contradict
Lemma~\ref{lem:crossedcontinua}. Small neighbourhoods of the ends
therefore lie in different complementary components. A connecting
continuum meets ${\wt B}$, and if it avoids its two {pseudo-leaf pieces} it meets ${\wt C}$.
\end{proof}

\subsection{Rotational horseshoes from 1-crossings}

We build a dense family of rotational horseshoes, in the sense of {Definition}~\ref{def:rotational_horseshoe}.

\begin{lemma}
	\label{lem:HPR_positioning}
	Let $P$ be an HPR from Proposition~\ref{prop:HPRdense}, and
	let $m\geq2$. A lift $\wt P$ and pairwise distinct
	$W_1,\ldots,W_m\in\Gamma$ can be chosen so that $\xi^s_0$
	separates the {stable side classes} of $\wt P$, and $\xi^u_0$
	separates the {unstable side classes} of every $W_l\wt P$.
\end{lemma}
\begin{proof}
	Start with any lift $\wt P_0$ with {stable side classes} $\alpha_1,\alpha_2$.
	Lemma~\ref{lem:separating_translate} gives a translate
	$T\xi^s_0$ separating them. {Define $\wt P = T^{-1}\wt P_0$.}
	
	Let $\beta_1,\beta_2$ be the {unstable side classes} of this lift.
	Choose a nondegenerate short transverse arc ${\wt\tau}$ in a regular
	foliation box all of whose {unstable classes} lie in the open
	segment $(\beta_1,\beta_2)$ of the unstable leaf tree, {as in Lemma \ref{lem:separating_translate}}.
	Each half of the {projected} invariant unstable leaf $L^u_0$ is dense, and hence $L^u_0$ 
	crosses infinitely many translates of ${\wt\tau}$. A fixed translate of ${\wt\tau}$ is
	crossed at most once. Thus the infinitely many crossings give infinitely many distinct deck labels.
	
	If $L^u_0$ crosses $W{\wt\tau}$, then $W^{-1}\xi^u_0$ belongs to
	$(\beta_1,\beta_2)$. Lemma \ref{lem:persistent_separation} gives that $\xi^u_0$
	separates $W\beta_1,W\beta_2$. {To conclude, select any $m$ deck transformations satisfying this separation.}	
\end{proof}

Figure~\ref{fig:horseshoe_configuration} illustrates these positioning
conditions and the boundary dynamics.

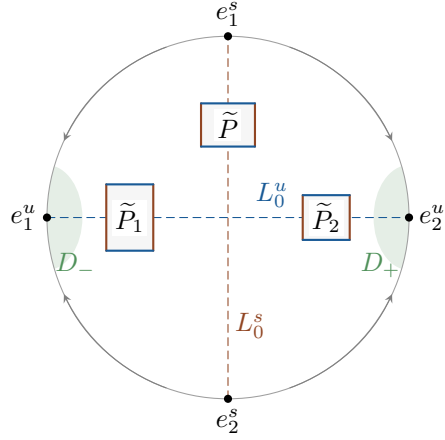
\begin{figure}[tbp]
	\centering
	\begin{tikzpicture}[scale=1.02,>=stealth,
		every node/.style={font=\small},line cap=round,line join=round]
		\definecolor{stablefig}{RGB}{145,70,35}
		\definecolor{unstablefig}{RGB}{25,90,150}
		\definecolor{endfig}{RGB}{80,145,90}
		\begin{scope}
			\clip (0,0) circle (2.35);
			\fill[endfig!15] (-2.35,0) ellipse (.46 and .68);
			\fill[endfig!15] (2.35,0) ellipse (.46 and .68);
		\end{scope}
		\draw[gray!75] (0,0) circle (2.35);
		\draw[stablefig,densely dashed] (0,-2.35)--(0,2.35);
		\draw[unstablefig,densely dashed] (-2.35,0)--(2.35,0);
		\foreach \x/\y in {0/2.35,0/-2.35,-2.35/0,2.35/0}
		\fill (\x,\y) circle (1.5pt);
		\node[above] at (0,2.35) {$e^s_1$};
		\node[below] at (0,-2.35) {$e^s_2$};
		\node[left] at (-2.35,0) {$e^u_1$};
		\node[right] at (2.35,0) {$e^u_2$};
		\draw[->,gray] (70:2.35) arc (70:25:2.35);
		\draw[->,gray] (110:2.35) arc (110:155:2.35);
		\draw[->,gray] (-70:2.35) arc (-70:-25:2.35);
		\draw[->,gray] (-110:2.35) arc (-110:-155:2.35);
		\fill[gray!7] (-.35,.92) rectangle (.35,1.48);
		\draw[stablefig,thick] (-.35,.92)--(-.35,1.48);
		\draw[stablefig,thick] (.35,.92)--(.35,1.48);
		\draw[unstablefig,thick] (-.35,.92)--(.35,.92);
		\draw[unstablefig,thick] (-.35,1.48)--(.35,1.48);
		\node[fill=white,inner sep=1pt] at (0,1.2) {$\widetilde P$};
		\fill[gray!7] (-1.58,-.43) rectangle (-.97,.43);
		\draw[stablefig,thick] (-1.58,-.43)--(-1.58,.43);
		\draw[stablefig,thick] (-.97,-.43)--(-.97,.43);
		\draw[unstablefig,thick] (-1.58,-.43)--(-.97,-.43);
		\draw[unstablefig,thick] (-1.58,.43)--(-.97,.43);
		\node[fill=white,inner sep=1pt] at (-1.275,0) {$\widetilde P_1$};
		\fill[gray!7] (.97,-.29) rectangle (1.58,.29);
		\draw[stablefig,thick] (.97,-.29)--(.97,.29);
		\draw[stablefig,thick] (1.58,-.29)--(1.58,.29);
		\draw[unstablefig,thick] (.97,-.29)--(1.58,-.29);
		\draw[unstablefig,thick] (.97,.29)--(1.58,.29);
		\node[fill=white,inner sep=1pt] at (1.275,0) {$\widetilde P_2$};
		\node[stablefig,right,fill=white,inner sep=1pt] at (.05,-1.4)
		{$L^s_0$};
		\node[unstablefig,above] at (.57,0) {$L^u_0$};
		\node[endfig] at (-1.98,-.62) {$D_-$};
		\node[endfig] at (1.98,-.62) {$D_+$};
	\end{tikzpicture}
	\caption{Positioning the source pseudo-rectangle $\widetilde P$ and
		two return copies $\widetilde P_l=W_l\widetilde P$.
		The invariant stable {class} separates the {classes} of the source
		stable sides, and the invariant unstable {class} separates the
		{classes} of the unstable sides of each target. The model leaves
		are dashed; brown sides are stable
		and blue sides are unstable. Boundary orbits approach the unstable
		ends, with the end neighbourhoods $D_-,D_+$ shown in green.
		The continua and their shapes are schematic.}
	\label{fig:horseshoe_configuration}
\end{figure}

{
\begin{proposition}
	\label{prop:dense_symbolic_horseshoes}
	Let $\wt P$ be an HPR, with $W_1,\ldots,W_m$, $\xi^s_0$ and $\xi^u_0$ {as in Lemma \ref{lem:HPR_positioning}}.
	Then, for every sufficiently large $j$, there is one marked rectangle
	$\wt R\subset\wt P$ such that {$G^{j}\wt R \dagger W_l{\wt R}$ for every $l$}. In particular, {$\wt\pi(\wt R)$ contains a rotational horseshoe for $g^{nj}$}.
\end{proposition}
}

{	
\begin{proof}
{
Write $\wt P_l=W_l\wt P$. By Lemma~\ref{lem:class_sides}, the two
unstable sides of each $\wt P_l$ lie in opposite complementary
sides of $\overline\eta^u(\xi^u_0)$. Choose open topological
half-planes $H_{2,l},H_{4,l}$, bounded by disjoint proper lines,
containing these compact sides, with compactified closures disjoint
from $\overline\eta^u(\xi^u_0)$. Such half-planes are obtained by
joining disk neighbourhoods of the compact sides to boundary
subarcs within their respective complementary components and
taking regular neighbourhoods. Choose small half-plane
neighbourhoods $H_1,H_3$ of the two attracting ends, disjoint from
all the $\wt P_l$ and from the closures of $H_{2,l},H_{4,l}$.
For each $l$ these four half-planes occur cyclically.

For all sufficiently large $j$, Lemma~\ref{lem:side_escape} puts
the stable sides of $G^j\wt P$ in $H_1,H_3$, respectively, and
Lemma~\ref{lem:compact_confinement} gives
\[
 G^j\wt P\cap\bigcup_{l=1}^m
 (\overline H_{2,l}\cup\overline H_{4,l})=\varnothing.
\]
Fix such a $j$. Lemma~\ref{lem:RectangleNearPseudoRectangle} gives
one sufficiently close marked rectangle $\wt R\subset\wt P$,
with horizontal sides near the unstable sides and vertical sides
near the stable sides. Choose the approximation close enough for
$G^j$ and the finitely many maps $W_l$ simultaneously. Then the
horizontal sides of $W_l\wt R$ lie in $H_{2,l},H_{4,l}$ and
$W_l\wt R$ avoids $H_1\cup H_3$, whereas the vertical sides of
$G^j\wt R$ lie in $H_1,H_3$ and $G^j\wt R$ avoids
$H_{2,l}\cup H_{4,l}$.

Thus,
Lemma~\ref{lem:crossing_configuration} applies to the ordered pair
${(G^j\wt R,W_l\wt R)}$ for every $l$. Hence
$G^j \wt R \dagger W_l \wt R$ for every $l$. Since $G^j$
is a lift of $g^{nj}$, Definition~\ref{def:rotational_horseshoe}
and Proposition~\ref{prop:realization_rotational_horseshoe} give the conclusion.
}
\end{proof}

}

\begin{corollary}
\label{cor:denseperiodic}
Every nonempty open set contains a {rotational horseshoe} {for an iterate of $g$}
and a periodic point. In particular,
$\overline{\operatorname{Per}(g)}=S$.
\end{corollary}
\begin{proof}
Proposition~\ref{prop:HPRdense} supplies an HPR inside the
prescribed open set. Position a source lift and finitely many
target translates by Lemma~\ref{lem:HPR_positioning}. Apply
Proposition~\ref{prop:dense_symbolic_horseshoes} and realize a
constant periodic itinerary.  The rectangle and its invariant
set project into the prescribed open set.
\end{proof}

\subsection{Intersections of continua and topological mixing}
\label{subsec:mixing}

The separation argument gives a statement for individual continua
that will also be used to locate heteroclinic intersections.

\begin{proposition}
\label{prop:continuum_mixing}
There is $N\geq1$ such that, for every pair of nontrivial compact
stable and unstable continua $K^s,K^u\subset S$, respectively,
there is $j_0$ with
\[
 g^{Nj}(K^u)\cap K^s\ne\varnothing
 \qquad\text{for every }j\geq j_0.
\]
Here a stable or unstable continuum downstairs is the projection
of a compact stable or unstable continuum in $\widetilde S$.
\end{proposition}
\begin{proof}
Use the return lift $G$ of $g^N$ and its invariant regular
{reference classes} $\xi^s_0,\xi^u_0$ fixed above. Choose a compact
lift $\widetilde K^u$ of $K^u$. It has two distinct stable
{classes}, since a continuum in both a stable and an unstable {pseudo-leaf}
would be a point. By Lemma~\ref{lem:separating_translate}, a
translate of the regular stable leaf defining $\xi^s_0$ separates
two such {classes}. Translate $\widetilde K^u$ back, keeping its
projection unchanged, so that $\xi^s_0$ separates them. Choose
$x_1,x_2\in\widetilde K^u$ on the two stable sides.
Lemma~\ref{lem:side_escape} gives
\[
 G^jx_1\longrightarrow e^u_1,\qquad
 G^jx_2\longrightarrow e^u_2.
\]

Similarly translate a compact lift $\widetilde K^s$ of $K^s$
so that $\xi^u_0$ separates two of its {unstable classes},
$\zeta_1,\zeta_2$. Choose regular {leaf classes} $\beta_i$ in
the open segments $(\zeta_i,\xi^u_0)$ of the dual unstable
tree, using Lemma~\ref{lem:separating_translate}. Each separates
$\zeta_1$ from $\zeta_2$, so $\widetilde K^s$ meets both
corresponding {pseudo-leaves} by Lemma~\ref{lem:intermediate}.

By Lemma~\ref{lem:crossing_band}, the continuum
\[
 {\wt B}=\overline\eta^u(\beta_1)\cup\widetilde K^s
       \cup\overline\eta^u(\beta_2)
\]
separates $e^u_1$ from $e^u_2$. Its two {pseudo-leaf pieces} are compact
and disjoint from $\overline\eta^u(\xi^u_0)$.
Lemma~\ref{lem:compact_confinement} therefore makes
$G^j\widetilde K^u$ avoid both pieces for all sufficiently
large $j$. This connected set contains points approaching the
two opposite ends. The band lemma consequently forces
$G^j\widetilde K^u\cap\widetilde K^s\ne\varnothing$.
Projection gives the assertion.
\end{proof}

\begin{lemma}\label{lem:powermixing}
If $h^N$ is topologically mixing for some $N\geq1$, then $h$ is
topologically mixing.
\end{lemma}
\begin{proof}
For nonempty open $U,V$ apply mixing of $h^N$ to each pair
$(h^rU,V)$, $0\leq r<N$, and take the maximum of the finitely many
thresholds. Writing every large integer as $Nj+r$ proves the assertion.
\end{proof}

\begin{proposition}\label{prop:mixing}
Every infinitely continuum-wise expansive element of $\Homeo_A(S)$,
for a closed pseudo-Anosov $A$, is topologically mixing.
\end{proposition}
\begin{proof}
Every nonempty open set contains nontrivial compact stable and
unstable continua, by Lemma~\ref{lem:kato}. Apply
Proposition~\ref{prop:continuum_mixing} to continua in two
prescribed open sets. It gives mixing of $g^N$, and
Lemma~\ref{lem:powermixing} gives mixing of $g$.
\end{proof}

\begin{corollary}\label{cor:icemixing}
{The essential factor in the closed pseudo-Anosov case is topologically mixing.}

\end{corollary}

\subsection{Unique homoclinic class}
\label{subsec:heteroclinic}

\begin{definition}[Topologically transverse intersection]
\label{def:transverse_connections}
Two sets $E,F\subset S$ have a \emph{topologically transverse
intersection in an open set $U$} if there is a closed Jordan rectangle
$R\subset U$, with sides $L,B,R',T$ in cyclic order, and continua
$C\subset E\cap R$ and $D\subset F\cap R$ such that $C$ joins
$L$ to $R'$ and avoids $B\cup T$, while $D$ joins $B$ to $T$
and avoids $L\cup R'$.
\end{definition}

{
\begin{definition}[Connected stable set]
	For $\wt p \in \wt S$, we define its \emph{connected stable set} $\widetilde{\mathcal C}^s(\widetilde p)$ {by}
	\[
	\begin{aligned}
		\widetilde{\mathcal C}^s(\widetilde p)
		&=\bigcup\{{\wt K}\subset\widetilde S:
		{\wt K\text{ is a stable continuum}},\ \widetilde p\in {\wt K}\},\\
		\mathcal C^s(p)&=\widetilde\pi
		\bigl(\widetilde{\mathcal C}^s(\widetilde p)\bigr).
	\end{aligned}
	\]
	{We use the symmetric unstable notation.} The projected set is independent
	of the chosen lift of $p$.
\end{definition}
}

Every compact stable continuum meeting
$\widetilde{\mathcal C}^s(\widetilde p)$ is contained in it:
adjoin that continuum to one joining the common point to
$\widetilde p$. We use this observation when taking crossing
subcontinua.

\begin{lemma}
\label{lem:periodic_continuum_sets}
For every periodic point $p$ of $g$, both $\mathcal C^s(p)$ and
$\mathcal C^u(p)$ are dense in $S$.
\end{lemma}
\begin{proof}
Let $r$ be the period of $p$. By Lemma~\ref{lem:kato}, choose
a nontrivial compact stable continuum $\widetilde K_p$ through
$\widetilde p$, and put $K_p=\widetilde\pi(\widetilde K_p)$.
Given a nonempty open $U$, choose a nontrivial compact unstable
continuum $K^u\subset U$. Proposition~\ref{prop:continuum_mixing}
gives $K^u\cap g^{-Nj}K_p\ne\varnothing$ for all sufficiently
large $j$. Choose such $j$ with $r\mid Nj$. A lift of
$g^{-Nj}K_p$ through $\widetilde p$ is a compact stable continuum
containing $\widetilde p$, hence
$g^{-Nj}K_p\subset\mathcal C^s(p)$. This proves density.
Applying the same argument to $g^{-1}$ proves unstable density.
\end{proof}

\begin{proposition}
\label{prop:dense_heteroclinic}
For every $p,q\in\Per(g)$, the sets $\mathcal C^s(p)$ and
$\mathcal C^u(q)$ have {a} topologically transverse intersection in
every nonempty open set $U$.
\end{proposition}

{
\begin{proof}
	As in Proposition \ref{prop:HPRdense}, we follow the construction of \cite[Proposition 8.11, {Steps} 1--4]{garciatal24chaotic}. Take a small open annulus $\overline{\wt U} \backslash \wt U'$: {apply the argument of Lemma~\ref{lem:local_crossing_supply} to translates of $\widetilde{\mathcal C}^s(\wt p)$, since its projection is dense by Lemma~\ref{lem:periodic_continuum_sets}. A finite union of the containing pseudo-leaves is closed with empty interior, so one can select infinitely many distinct stable classes. Boundary bumping then gives} a sequence of pairwise disjoint minimal crossing stable continua $\wt K^s_n$ respectively contained in distinct translates of $\wt{\mathcal{C}}^s(\wt p)$, {with uniform forward diameter bounds by Lemma~\ref{lem:stable_continua_equiv}. Extract} a Hausdorff limit $\wt K^s$ which is a stable continuum and can be assumed disjoint from all $\wt K^s_n$. {Use translates of $\widetilde{\mathcal C}^u(\wt q)$ for the unstable continua, by the same argument.} We repeat the construction of \cite[Proposition 8.11, Steps 3 and 4]{garciatal24chaotic}, where Lemma \ref{lem:HPR_thin_strip} {gives infinitely many topologically} {transverse intersections} inside the annulus, and therefore inside $U$. 
\end{proof}
}

\section{The conservative model}\label{sec:conservative}\label{sec:PARF}

\subsection{A good invariant measure and conservative conjugacy}

\begin{proposition}\label{prop:AreaPreserving}
The essential factor $g$ of a homeomorphism isotopic to a pseudo-Anosov map preserves an atomless probability with
full support. Thus it is conjugate, by a homeomorphism
isotopic to the identity, to a homeomorphism preserving a
smooth probability area measure.
\end{proposition}

\begin{proof}
Let $\{U_i\}_{i\geq1}$ be a countable base of nonempty open sets.
By Proposition~\ref{prop:dense_symbolic_horseshoes}, for each $i$ choose a compact $g^{k_i}$-invariant horseshoe
$\Lambda_i\subset U_i$ and a continuous factor map from
$g^{k_i}|_{\Lambda_i}$ onto the two-sided full shift. Lift its fair
Bernoulli measure to a probability on $\Lambda_i$, average this
measure under $g^{k_i}$, and take a {weak--$\ast$} limit. The resulting
invariant measure $\nu_i$ still projects to Bernoulli measure.
It has no atoms, because the image of an atom under the factor map
would be an atom of Bernoulli measure.
Set
\[
 \mu_i=\frac1{k_i}\sum_{r=0}^{k_i-1}(g^r)_*\nu_i,
 \qquad \mu=\sum_{i=1}^\infty 2^{-i}\mu_i.
\]
Every $\mu_i$ is $g$-invariant and atomless, with
$\mu_i(U_i)\geq1/k_i$. Therefore $\mu$ is invariant, atomless and
has full support. The Oxtoby--Ulam theorem, with its identity-isotopy
version on compact surfaces, gives the final conjugacy; see
\cite{oxtobyulam41measure,alpernprasad00typical}.

\end{proof}

\subsection{Proof of the theorem in the pseudo-Anosov case}
\begin{proof}[Proof of Theorem~\ref{thmA:pA}]
Proposition~\ref{prop:splitting} gives the upper semicontinuous
cellular decomposition and its uniform lifted diameter bound.
Remarks~\ref{rem:canonical_model} and \ref{rem:isotopic_factor}
identify the quotient with $S$ so that $q$ is homotopic to the
identity and $g$ is isotopic to $f$.
Lemma~\ref{lem:essential_factor_ice}, Proposition~\ref{prop:Tight},
and Proposition~\ref{prop:mixing} prove infinite continuum-wise
expansiveness, tightness, and mixing.
Proposition~\ref{prop:dense_symbolic_horseshoes} and
Corollary~\ref{cor:denseperiodic} give the local {rotational horseshoes}
and density of periodic points.
Proposition~\ref{prop:dense_heteroclinic} gives the {unique homoclinic class}. Proposition~\ref{prop:AreaPreserving}
gives the invariant probability and the conservative conjugacy.
The decomposition is defined by the dynamically bounded continua
of $f$, and its idempotence is Remark~\ref{rem:idempotent}.
\end{proof}

Let $f$ be isotopic, relative to a finite invariant set $X$, to
a pseudo-Anosov representative $A$. We remove the one-prong
singularities by a finite regular branched covering. The essential
decomposition upstairs then descends to the original surface.
The main point is to identify this descended decomposition
intrinsically: dynamical unboundedness of a continuum upstairs
must remain visible after projection to the closed surface.

\subsection{The branched covering and dynamical metric control}

\begin{proposition}\label{prop:PARFcover}
There is a finite regular branched covering $b:\wh S\to S$, branched
only over $X$, and compatible lifts $\wh f,\wh A$ with
$b\wh f=fb$, $b\wh A=Ab$, such that $\wh f$ is isotopic to the
genuine pseudo-Anosov map $\wh A$ on the closed surface $\wh S$.
If $X$ is empty no covering is required.
\end{proposition}
\begin{proof}
{
Assume $X\ne\varnothing$. The group $\pi_1(S\setminus X)$ is a
finitely generated free group. For each peripheral loop about a
one-prong singularity choose a finite-index normal subgroup omitting
it. Intersect these subgroups, and then intersect all subgroups of
index at most the resulting index. There are finitely many such
subgroups; their intersection is characteristic, finite-index,
normal, and still omits every selected peripheral loop.
The associated regular cover of $S\setminus X$ extends over the
punctures to a finite branched covering. The local degree at a
one-prong point is at least two, so all lifted singularities have
at least two prongs. The pulled-back invariant foliations make $\wh A$
a genuine pseudo-Anosov map.}
\end{proof}

\begin{lemma}\label{lem:PARFmetric}
Let $b:\Sigma\to S$ be a {continuous} map of closed surfaces and let
$B:\widetilde\Sigma\to\widetilde S$ be an equivariant lift.
For every $R<\infty$ there is $J_B(R)<\infty$ such that
\[
 d_{\widetilde\Sigma}(x,y)\leq R
 \quad\Longrightarrow\quad
 d_{\widetilde S}(Bx,By)\leq J_B(R).
\]
\end{lemma}
\begin{proof}
Translate $x$ into a compact fundamental set. The corresponding
translate of $y$ belongs to its compact $R$-neighbourhood.
Continuity bounds the diameter of the image of that neighbourhood,
and equivariance makes the same bound valid for the original pair.
\end{proof}

\subsection{Projected dynamical growth}\label{subsec:visible_growth}

\begin{lemma}\label{lem:leaf_in_infinite_cover}
Let $\mathcal F$ be an invariant foliation of a closed pseudo-Anosov
surface $\Sigma$, and let $p:Y\to\Sigma$ be a connected infinite
covering. No regular leaf of the lifted foliation has compact
closure in $Y$.
\end{lemma}
\begin{proof}
Suppose a lifted regular leaf has compact closure $E$.
Write $\Sigma_0=\Sigma\setminus\sing(\mathcal F)$ and
$Y_0=p^{-1}(\Sigma_0)$. {The set $E\cap Y_0$ is saturated by
nonsingular leaf paths: follow points of the original leaf through flow boxes.
Compactness gives a finite uniform bound for
\[
 m(x)=\#(E\cap p^{-1}(x)),\qquad x\in\Sigma_0.
\]
This integer-valued function is upper semicontinuous, using evenly covered neighbourhoods and compactness of $E$.
It is constant along any nonsingular leaf path, by construction.} 

Choose a nonsingular point where the positive maximum $m_0$ of
$m$ is attained. {At least one half-leaf through this point avoids
the singularities and is dense in $\Sigma$ (Lemma \ref{lem:minimal_filling_foliations}). The function equals $m_0$ along
that half-leaf. The set $\{m\geq m_0\}$ is closed in $\Sigma_0$,
so $m = m_0$ is constant in $\Sigma_0$.} This constant multiplicity makes $E\cap Y_0$ open in $Y_0$.
Indeed, over $x\in\Sigma_0$ choose disjoint local sheets about
the $m_0$ points of $E\cap p^{-1}(x)$. After shrinking the base
neighbourhood, compactness excludes all other sheets from $E$;
constant multiplicity then fills each of the chosen sheets.

The set $E\cap Y_0$ is also closed and nonempty, {and the surface
$Y_0$ is connected}. Thus
$Y_0\subset E$, and closedness gives $Y=E$. This contradicts
noncompactness of an infinite cover of
$\Sigma$.
\end{proof}

\begin{lemma}\label{lem:visible_leaf}
{Let $\Sigma$ be a hyperbolic surface,} $b:\Sigma\to S$ be a finite branched covering of a closed {
surface $S$ of genus $g\ge 1$}, and let
$B:\widetilde\Sigma\to\tS$ be a compatible lift.
If $L$ is a regular lifted invariant leaf of a closed
pseudo-Anosov map on $\Sigma$, then $B(L)$ is unbounded in $\tS$.
\end{lemma}
\begin{proof}
Take the connected component $Y$ of the pullback of
$b:\Sigma\to S$ along $\tS\to S$ determined by $B$.
It is a covering $p:Y\to\Sigma$ corresponding to $\ker b_*$,
and the other projection $\beta:Y\to\tS$ is a proper finite
branched covering. Properness follows from the pullback construction. The map $\beta$ is
open and closed and hence surjective. Thus $Y$ is noncompact,
and $p$ is an infinite covering.

Factor $B$ as $\widetilde\Sigma\xrightarrow{r}Y
\xrightarrow{\beta}\tS$. If $B(L)$ were bounded, properness of
$\beta$ would put $r(L)$ in a compact subset of $Y$.
But $r(L)$ is a regular leaf of the lifted foliation, contradicting
Lemma~\ref{lem:leaf_in_infinite_cover}.
\end{proof}

\begin{proposition}\label{prop:visible_dynamical_growth}
	Let ${b}:\Sigma\to S$ be as above, let $h:\Sigma\to\Sigma$ be
	isotopic to a closed pseudo-Anosov map, and suppose ${bh=fb}$.
	Choose compatible lifts $H,F,B$ with $BH=FB$.
	{Then, a continuum $\wt Q\subset\widetilde\Sigma$ is dynamically bounded for $H$ {if and only if} its projection $B \wt Q$ is dynamically bounded for $F$.} 
\end{proposition}
\begin{proof}
	{Suppose $\wt Q$ is dynamically bounded. Then, Lemma \ref{lem:PARFmetric} gives {uniformly bounded diameters} for $BH^j \wt Q = F^jB \wt Q$, and thus shows that $B \wt Q$ is dynamically bounded. 
	
	Suppose that $\wt Q$ is dynamically unbounded. By Corollary~\ref{cor:dynamically_bounded_continua}, we can assume without loss of generality that {$\wt Q$} intersects two distinct stable pseudo-leaves}. As in Section \ref{sec:horseshoes}, choose a regular periodic point of
	the pseudo-Anosov representative and its lift {$\wt p_A$}, with {corresponding lift} $G=\gamma_0H^n$ and invariant {leaf classes}
	$\xi^s_0,\xi^u_0$. Pass to a further power to fix the prongs.
	Lemma~\ref{lem:separating_translate} gives a deck translate {$\wt Q'$} of
	{$\wt Q$} whose stable {classes} lie on both sides of $\xi^s_0$.
	
	Let $L^u_0$ be the unstable leaf through {$\wt p_A$} and put $R=2C+1$.
	There is a constant $J_B(R)$ such that
	\[
	d_{\widetilde\Sigma}(x,y)\leq R
	\quad\Longrightarrow\quad d_{\tS}(Bx,By)\leq J_B(R).
	\]
	This is the upper control from Lemma~\ref{lem:PARFmetric},
	using equivariance and compactness of a fundamental neighbourhood.
	By Lemma~\ref{lem:visible_leaf}, choose $v_k\in L^u_0$ with
	$d_{\tS}(Bv_k,{B\wt p_A})\to\infty$.
	Lemma~\ref{lem:coarse_leaf_tracing}, {applied to $v_k$ and $\wt p_A$,
	provides, for every sufficiently large $j$, points in the same $G^j\wt Q'$
	respectively within distance $R$ of $v_k$ and $\wt p_A$}. Consequently,
	\begin{equation}\label{eq:visible_growth}
		\diam(F^{nj}B{\wt Q}) = \diam(BG^j{\wt Q}) = \diam(BG^j{\wt Q'})\geq d_{\tS}(Bv_k,B{\wt p_A})-2J_B(R).
	\end{equation}
	Thus $\diam(F^{nj}B \wt Q)\to\infty$ as $j\to\infty$: $B \wt Q$ is dynamically unbounded. 
\end{proof}

{We obtain the following result as a consequence.}

\begin{corollary}\label{cor:finite_factor_ice}
Under the hypotheses of Proposition~\ref{prop:visible_dynamical_growth},
if $h$ is infinitely continuum-wise expansive, then so is $f$.
\end{corollary}

\subsection{Descending the cellular decomposition}

\begin{lemma}\label{lem:finite_cellular_descent}
Let a finite group $H$ act by orientation-preserving homeomorphisms
on a closed orientable surface $\Sigma$, and let $\mathcal D$ be an
$H$-invariant cellular decomposition. If $b:\Sigma\to\Sigma/H$ is
the quotient, then the distinct sets $b(C)$, $C\in\mathcal D$,
form a cellular decomposition. If $\mathcal D$ is upper
semicontinuous, so is this decomposition.
\end{lemma}
\begin{proof}
If $b(C)$ meets $b(C')$, then $C$ meets a translate of $C'$, so
they are the same decomposition element up to $H$. The images thus
form a partition by continua. Fix $C$ and its stabilizer $H_C$ {---if it is trivial, then $b$ is a homeomorphism in a neighbourhood of $C$ and the assertion follows---}.
Cellularity gives arbitrarily small disk neighbourhoods $D$ of $C$.
Choose $D$ so that all $hD$, $h\in H_C$, lie in a common disk
neighbourhood of $C$, while $D$ is disjoint from its translates
outside $H_C$. The component containing $C$ of
$\bigcap_{h\in H_C}hD$ is an $H_C$-invariant open disk: in that
common planar chart the complements of the finitely many disks
are connected and share the exterior, so each intersection component
is simply connected. To see that its quotient is a disk, extend the finite action to the
prime-end closed disk. The boundary action is faithful: a periodic
disk homeomorphism equal to the identity on the boundary is the
identity, by \cite[Proposition 3.2]{constantinkolev94periodic}.
A finite orientation-preserving circle group is cyclic, and its
generator is conjugate on the disk to a rotation by
\cite[Theorem 3.1]{constantinkolev94periodic}. The interior quotient
is therefore an open disk. Thus $b(C)$ has arbitrarily small disk
neighbourhoods and is cellular. {Upper semicontinuity follows directly from the definition}.
\end{proof}

{This allows us to project the essential factor through the branched covering.}

\begin{corollary}\label{cor:PARFcommute}
{
The essential cell-like semiconjugacy $\wh q$ of $\wh f$ descends to a}
cell-like semiconjugacy $q_b:S\to S_b$ with uniformly bounded
lifted fibres. Its factor $g_b$ is isotopic to $f$ after marking
$S_b$ as $S$, and the diagram satisfies
\[
 q_b b=b_b\wh q,\qquad
 b_b\wh g=g_b b_b,\qquad q_b f=g_b q_b,
\]
where $b_b$ is a finite branched covering and $\wh g$ is the
essential factor upstairs. Every fibre of $q_b$ is dynamically
bounded for $f$.
\end{corollary}

\begin{proof}
	{Lemma \ref{lem:finite_cellular_descent} gives the descent of the cell-like semiconjugacy. The last assertion follows from Lemma \ref{lem:PARFmetric}.}
\end{proof}

\begin{proposition}\label{prop:PARFchaotic}
The factor $g_b$ is topologically mixing, infinitely continuum-wise
expansive, and tight. It has local {rotational} horseshoes with
periodic realization, dense periodic points, and an atomless
full-support invariant probability. For every pair of periodic
points, the connected stable and unstable sets have a topologically
transverse intersection in every nonempty open set. 
\end{proposition}
\begin{proof}
Mixing passes from $\wh g$ to its continuous surjective factor
$g_b$, by Proposition~\ref{prop:mixing}.

Apply Corollary~\ref{cor:finite_factor_ice} to the finite
branched covering $b_b$ and the essential factor $\wh g$ upstairs.
The latter is infinitely cw-expansive and remains in the closed
pseudo-Anosov isotopy class, so $g_b$ is infinitely cw-expansive.
In particular it is ordinarily cw-expansive, and
Proposition~\ref{prop:Tight} gives tightness.

For the horseshoe assertion, choose an open disk inside the given
open set avoiding the finite branch set, and one inverse component
on which $b_b$ is a homeomorphism.
Proposition~\ref{prop:dense_symbolic_horseshoes} supplies a compact
horseshoe for an iterate inside this component. It is conjugate to its projected
system, since all of its return iterates stay inside it{.} 
The measure assertion follows either from
Proposition~\ref{prop:AreaPreserving}, or by pushing
forward a good invariant measure of $\wh g$: finiteness of fibres
preserves atomlessness and surjectivity preserves full support.

For the homoclinic-class assertion, let $p,q\in\Per(g_b)$ and let
$U\subset S_b$ be nonempty and open. Choose $r\geq1$ with
$g_b^r(p)=p$ and $g_b^r(q)=q$. The homeomorphism $\wh g^r$
permutes each of the finite fibres $b_b^{-1}(p)$ and
$b_b^{-1}(q)$. Thus any choices $\wh p\in b_b^{-1}(p)$ and
$\wh q\in b_b^{-1}(q)$ are periodic for $\wh g$. Choose an open disk $V\subset U$ avoiding the singular values
of $b_b$, and a component $\wh V$ of $b_b^{-1}(V)$ on which
$b_b$ is a homeomorphism. Apply
Proposition~\ref{prop:dense_heteroclinic} upstairs to
$\wh p,\wh q$ and $\wh V$. Its construction gives a rectangle
$\wh R\subset\wh V$ with transverse crossing stable and
unstable continua $\wh C^s\subset\mathcal C^s(\wh p)$ and
$\wh C^u\subset\mathcal C^u(\wh q)$. {Projecting gives the desired transverse intersections}. 

\end{proof}

\begin{proposition}\label{prop:PARFsplitting}
Every PARF homeomorphism $f$ satisfies $(\Bnd)$. The fibres of
$q_b$ are exactly its maximal dynamically bounded continua, so
the intrinsic essential factor is $g_b$.
\end{proposition}
\begin{proof}
Mark $q_b$ as homotopic to the identity and take its equivariant
lift $\Phi_b$ at bounded distance from the identity, as in
\eqref{eq:Phibounded}. If {$\wt Q\subset\tS$} is a dynamically bounded
continuum for a lift of $f$, then {$\Phi_b(\wt Q)$} is dynamically
bounded for the compatible lift of $g_b$. The latter is infinitely
cw-expansive by Proposition~\ref{prop:PARFchaotic}, so
{$\Phi_b(\wt Q)$} is a point. Thus {$\wt Q$} lies in a lifted fibre of $q_b$.
The uniform fibre bound of Corollary~\ref{cor:PARFcommute}
proves $(\Bnd)$, with
\[
 M_f\leq J_{\widetilde b}(M_{\wh f}).
\]
Every fibre is itself dynamically bounded. Therefore the
intrinsic {decomposition elements} and the fibres of $q_b$ coincide.
\end{proof}

\begin{proof}[Proof of Theorem~\ref{thmB:PARF}]
Proposition~\ref{prop:PARFcover} supplies the regular branched
lift. Corollary~\ref{cor:PARFcommute} gives the descended
cellular factor, and Proposition~\ref{prop:PARFchaotic} gives
its infinite continuum-wise expansiveness, tightness, and
mixing, local {rotational} horseshoes and a good invariant measure.
Periodic realization gives density of periodic points, and
the Oxtoby--Ulam argument gives the conservative conjugacy.
Proposition~\ref{prop:PARFsplitting} identifies it
with the intrinsic quotient.
\end{proof}

\subsection{Uniformly bounded periodic islands}
\label{subsec:periodic_islands}

A \emph{periodic island} here means any periodic open topological
disk $U\subset S$: there is $k\geq1$ with $f^k(U)=U$.
No regularity of its boundary or maximality among periodic disks
is required. Its \emph{{lifted diameter}} is
\[
 \mathcal D(U)=\diam(\widetilde U),
\]
where $\widetilde U$ is any connected lift. Deck invariance makes
this independent of the lift.

\begin{lemma}
\label{lem:inessential_fixed_original}
For every $f$ in Theorem~\ref{thmA:pA} or \ref{thmB:PARF},
the set $\Fix(f^k)$ is inessential for every $k\geq1$.
\end{lemma}
\begin{proof}
{Suppose that $Z = \Fix(f^k)$ is essential for some $k {\geq} 1$. {Since $Z$ is compact, there exists an essential connected component $Z'$ of $Z$, and} every connected component of the lift $\wt Z'$ to the universal covering is unbounded ---otherwise it would be compact and $Z'$ would be inessential---. By Lemma \ref{lem:large_compact_pieces}, this would give arbitrarily large dynamically bounded continua for $f$, which is a contradiction {since} $f$ satisfies (UB).}
\end{proof}

\begin{remark}[PARF and fully essential dynamics]
\label{rem:PARF_fully_essential}
Every identity-isotopic PARF homeomorphism has fully essential
dynamics in the sense of \cite{koropeckital18fully}. To see this,
\cite[Theorem A]{guiheneufmiliton23hyperbolic} gives
$\operatorname{int}\rho^{\rm erg}_{H_1}(f)\ne\varnothing$.
The Misiurewicz-Ziemian rotation set contains this set. It therefore
cannot be contained in a finite union of proper linear subspaces.
By \cite[Lemma 6.1]{koropeckital18fully}, this rules out a
homotopically bounded invariant connected open set which is
essential but not fully essential. By
\cite[Lemma 6.4]{koropeckital18fully}, it also rules out uniformly
bounded displacement. Finally,
Lemma~\ref{lem:inessential_fixed_original} rules out essential
fixed-point sets of iterates. Thus none of the three exceptional
alternatives used to define non-fully-essential dynamics in
\cite[Theorem C]{koropeckital18fully} occurs.
\end{remark}

\begin{proof}[Proof of Theorem~\ref{thmD:periodic_islands}]
Let $U$ be a periodic open disk and choose $k\geq1$ with
$f^k(U)=U$. The map $f^k$ is nonwandering and isotopic to the
identity, and Lemma~\ref{lem:inessential_fixed_original} makes
$\Fix(f^k)$ inessential. The restriction to the invariant open
disk $U$ is nonwandering, so $U$ has no wandering cross-cuts.
The invariant-disk theorem of Koropecki and Tal
\cite[Theorem A]{koropeckital18fully} therefore gives
$\mathcal D(U)<\infty$.

Let $F$ be the preferred lift of $f$, which commutes with every
deck transformation, and write $F^k\widetilde U=T\widetilde U$.
For $m\in\mathbb Z$ and $0\leq r<k$,
\[
 F^{mk+r}\widetilde U=T^mF^r\widetilde U.
\]
The finitely many sets $F^r\widetilde U$ have finite diameter:
$\overline{\widetilde U}$ is compact in the proper lifted metric.
Thus every compact arc in $\widetilde U$ is dynamically bounded.
By Proposition~\ref{prop:PARFsplitting}, its diameter is at most
the uniform bound $M_f$ for dynamically bounded continua.
Any two points of $\widetilde U$ can be joined by a compact arc
in $\widetilde U$. {Taking the supremum shows that $U$ is dynamically bounded, as well as its closure $\overline{U}$}.
\end{proof}

The {torus case was already settled} in \cite[Theorem E]{garciatal24chaotic}, {answering} \cite[Question G]{koropeckital13strictly} in the
rotation-interior setting. In higher genus,
\cite{koropeckital18fully} {gives a uniform bound for fixed islands} in the setting of fully essential dynamics. By
Remark~\ref{rem:PARF_fully_essential},
 the identity-isotopic PARF
maps considered here already belong to that setting. Under the
nonwandering hypothesis, Theorem~\ref{thmD:periodic_islands}
makes the diameter bound independent of the period.

\section{Rotational invariants of the quotient}\label{sec:rotational}

In this section $q:S\to S$ is any of the marked cellular
semiconjugacies constructed above: the essential quotient in the
closed pseudo-Anosov case, the descended PARF quotient $q_b$, or
the intrinsic PARF quotient identified in Proposition~\ref{prop:PARFsplitting}. Write
$qf=gq$ and let $\Phi$ be its equivariant lift. We use the cellularity and bounded displacement of the quotient
map to compare these invariants.

\subsection{Cell-like fibres and the fine curve graph}

Recall that a compact subset of a surface is cellular if it is the intersection of a nested sequence of closed disks; by \cite[Lemma 2.6]{fanonihenselleroux26translation} an inessential filled continuum of $S$ is cellular. {A continuous map between surfaces is \emph{cell-like} when its fibres are cellular.} Recall that the semiconjugacies from Theorems \ref{thmA:pA} and \ref{thmB:PARF} are cell-like.

Let $(\mathcal{N})\mathcal{C}^\dagger(S)$ be the fine (nonseparating) curve graph of $S$ and $\mathrm{tl}(\cdot)$ the stable translation length of the {induced} action of a homeomorphism, as in \cite{bowdenhenselwebb22quasimorphisms, fanonihenselleroux26translation}.

\begin{proposition}\label{prop:finecurve}
	Let $g$ be the essential factor of $f$. Then, $\mathrm{tl}(f) = \mathrm{tl}(g)$.
\end{proposition}

\begin{proof}
{
For the nonseparating graph, apply
\cite[Theorem~1.2(1)]{fanonihenselleroux26translation}.
For the full graph it is enough to use
\cite[Lemma~3.2]{fanonihenselleroux26translation}.
}
\end{proof}

\subsection{Homological rotation sets}

Choose smooth closed one-forms representing a basis of $H^1(S;\mathbb R)$.
Their integrals on the simply connected universal cover define a
continuous map $\mathcal A:\tS\to H_1(S;\mathbb R)$ with
\[
 \mathcal A(Tx)=\mathcal A(x)+[T].
\]
{
For $f\in\Homeo_0(S)$, fix an identity isotopy and let $F$ be
its lift starting at the identity. Then $F$ commutes with
every deck transformation. When $S$ has genus at least two,
it is the unique lift with this property. On the torus this
does not determine a unique lift, and we keep $F$ fixed.
Choose the compatible lift $G$ of $g$ satisfying
$\Phi F=G\Phi$.

The function $v_F(x)=\mathcal A(F\wt x)-\mathcal A(\wt x)$
is independent of the lift $\wt x$ of $x$ and is continuous.
Its orbit averages give the usual homological rotation sets;
replacing it by a bounded fundamental-domain cocycle changes
the sums only by a bounded coboundary. Denote the
Misiurewicz--Ziemian accumulation set by $\rho_{H_1}(F)$ and
the set of integrals over ergodic invariant probabilities by
$\rhoe_{H_1}(F)$. Use the analogous notation for $G$. For an equivalent topological definition, see \cite[Definition 1.3]{garciaguiheneuflessa23tracking}.
}

\begin{proposition}\label{prop:homologicalrotation}
{
For an identity-isotopic $f$, a marked cellular factor as above,
and the specified compatible lifts $F,G$,
\[
 \rho_{H_1}(F)=\rho_{H_1}(G),\qquad
 \rhoe_{H_1}(F)=\rhoe_{H_1}(G).
\]
}
\end{proposition}
\begin{proof}
{
Use the compatible lifts $F,G$ fixed above.
} The function
$u(x)=\mathcal A(\Phi\wt x)-\mathcal A(\wt x)$ is a well-defined
bounded continuous function on $S$. The semiconjugacy gives the
exact identity
\begin{equation}\label{eq:rotation_coboundary}
 {v_G(qx)=v_F(x)+u(fx)-u(x).}
\end{equation}
Hence corresponding orbit sums differ by at most $2\|u\|_\infty$.
Surjectivity of $q$ gives both inclusions of the accumulation sets.

An ergodic invariant measure for $f$ projects to an ergodic
invariant measure for $g$, and the coboundary has integral zero.
Conversely, lift an ergodic invariant probability $\nu$ for $g$
to any probability on $S$, average its iterates under $f$, and
take a weak limit. This produces an invariant lift $\mu$.
In its ergodic decomposition almost every component still
projects to $\nu$, since $\nu$ is an extreme invariant measure.
Equation~\eqref{eq:rotation_coboundary} gives the same rotation
vector for such a component. This proves the ergodic equality.
\end{proof}

\subsection{Quotient-homology rotation with a specified lift}

Let $L:H_1(S;\mathbb R)\to V$ be a linear quotient with
$L\circ f_*=L$. The canonical example is the projection onto
\[
 V=H_1(S;\mathbb R)/\operatorname{im}(f_*-\Id).
\]
Fix a lift $F$ of $f$ and the compatible lift $G$ of $g$ satisfying
$\Phi F=G\Phi$. Define
\[
 v_{F,L}(x)=L\bigl(\mathcal A(F\wt x)-\mathcal A(\wt x)\bigr).
\]
If $\wt x$ is replaced by $T\wt x$, the change is
$L(f_*[T]-[T])=0$. Thus this is a well-defined continuous cocycle.
Write $\rho_L(F)$ and $\rho_L^{\rm erg}(F)$ for its Misiurewicz-Ziemian
and ergodic-integral rotation sets.

\begin{proposition}\label{prop:transverserotation}
Under these hypotheses, 
\[
 \rho_L(F)=\rho_L(G),\qquad
 \rho_L^{\rm erg}(F)=\rho_L^{\rm erg}(G).
\]
If $f$ preserves a reduction multicurve, so does the isotopy
class of $g$.
\end{proposition}
\begin{proof}
Apply $L$ to the computation in
\eqref{eq:rotation_coboundary}, using
$u_L(x)=L(\mathcal A(\Phi\wt x)-\mathcal A(\wt x))$.
This gives the same bounded coboundary, and the Misiurewicz-Ziemian
and ergodic-lifting arguments apply verbatim. Preservation of
the multicurve is an isotopy-class statement and follows from
$g\simeq f$.
\end{proof}

\begin{remark}\label{rem:rotation_normalization}

For a deck transformation $T$, replacing $F$ by $TF$ gives
\[
 v_{TF,L}=v_{F,L}+L([T]),\qquad
 \rho_L(TF)=\rho_L(F)+L([T]),
\]
and the same translation formula holds for the ergodic
rotation set. Thus these sets are unchanged whenever
$L([T])=0$. This includes all
$T$ with $[T]\in\operatorname{im}(f_*-\Id)$. A reference lift must still be specified.
\end{remark}

\begin{proof}[Proof of Theorem~\ref{thmC:rotational}]
The assertions are Propositions~\ref{prop:finecurve},
\ref{prop:homologicalrotation}, and \ref{prop:transverserotation},
respectively.
\end{proof}

\bibliographystyle{amsalpha}
\bibliography{fchg_biblio}

\end{document}